\documentclass[12pt,reqno]{amsart}

\usepackage{amsmath,amssymb,ifthen}
\usepackage{fullpage,enumerate}
\usepackage{graphicx,psfrag,subfigure}
\usepackage{color}
\usepackage{moreverb}

\def\H{\widetilde{H}}

\def\R{{\mathbb R}}
\def\N{{\mathbb N}}

\def\EE{{\mathcal E}}
\def\HH{{\mathcal H}}
\def\MM{{\mathcal M}}
\def\OO{{\mathcal O}}
\def\PP{{\mathcal P}}

\def\SS{{\mathcal S}}
\def\TT{{\mathcal T}}
\def\XX{{\mathcal X}}
\def\diam{{\rm diam}}

\def\norm#1#2{\|#1\|_{#2}}
\def\seminorm#1#2{\vert #1\vert_{#2}}
\def\abs#1{\seminorm{#1}{}}
\def\enorm#1{|\hspace*{-.5mm}|\hspace*{-.5mm}|#1|\hspace*{-.5mm}|\hspace*{-.5mm}|}
\def\set#1#2{\big\{#1\,:\,#2\big\}}

\def\eps{\varepsilon}
\def\dual#1#2{\langle#1\,,\,#2\rangle}

\def\u{\mathbf{u}}
\def\U{\mathbf{U}}
\def\v{\mathbf{v}}
\def\V{\mathbf{V}}
\def\mesh{\EE}

\def\wat#1{\widehat #1}

\def\slp{\mathfrak{V}} % simple-layer potential
\def\dlp{\mathfrak{K}} % double-layer potential
\def\hyp{\mathfrak{W}} % hypersingular integral operator
\def\Verfuerth{Verf\"urth}

\newcounter{constantsnumber}
\def\namec#1#2{%
  \ifthenelse{\equal{#1}{inv}}{C_{\rm inv}}{%
  \ifthenelse{\equal{#1}{invtilde}}{\widetilde C_{\rm inv}}{%
  \ifthenelse{\equal{#1}{lipschitz}}{C_{\rm lip}}{%
  \ifthenelse{\equal{#1}{monotone}}{C_{\rm mon}}{%
  \ifthenelse{\equal{#1}{reliability}}{C_{\rm rel}}{%
  \ifthenelse{\equal{#1}{caccioppoli}}{C_{\rm cacc}}{%
  \ifthenelse{\equal{#1}{xxx}}{\widetilde C_{\rm near}}{%
  \ifthenelse{\equal{#1}{nearfield}}{C_{\rm near}}{%
  \ifthenelse{\equal{#1}{farfield}}{C_{\rm far}}{%
  \ifthenelse{\equal{#1}{cea}}{C_{\mbox{\scriptsize C\'ea}}}{%
  \ifthenelse{\equal{#2}{newcounter}}{\refstepcounter{constantsnumber}\label{const#1}}{}C_{\ref{const#1}}}%
}}}}}}}}}}
\def\setc#1{\namec{#1}{newcounter}}
\def\c#1{\namec{#1}{reference}}

\def\osc{{\rm osc}}

\newtheorem{theorem}{Theorem}
\newtheorem{proposition}[theorem]{Proposition}
\newtheorem{lemma}[theorem]{Lemma}
\newtheorem{corollary}[theorem]{Corollary}
\newtheorem{algorithm}[theorem]{Algorithm}

\newenvironment{remark}{\bigskip\noindent\textbf{Remark.}\ \it}{\qed\bigskip}

\numberwithin{equation}{section}

\newtheorem{question}{Open Question}

\def\subsection#1
{
 \bigskip

 \refstepcounter{subsection}
 {\noindent\bf\arabic{section}.\arabic{subsection}.~#1.~}
}

\def\revision#1{#1}

\def\interior{{\mathrm{int}}}
\def\frakA{{\mathfrak{A}}}
\def\OO{{\mathcal O}}

\def\div{{\rm div}}
\def\i{{\rm int}}
\def\e{{\rm ext}}
\def\normal{{\boldsymbol\nu}}
\def\mfrac#1#2{\mbox{$\frac{1}{2}$}}
\def\loc{{\ell oc}}
\def\sym{{\rm sym}}

\def\JJ{\mathcal J}

\begin{document}%%%%%%%%%%%%%%%%%%%%%%%%%%%%%%%%%%%%%%%%%%%%%%%%%%%%%

\title[Inverse Estimates for Elliptic Boundary Integral Operators]%
{Inverse Estimates for Elliptic Boundary Integral Operators and Their Application to the Adaptive Coupling of FEM and BEM}
\date{\today}

\author{M.~Aurada}
\author{M.~Feischl}
\author{T.~F\"uhrer}
\author{M.~Karkulik}
\author{J.\ M.~Melenk}
\author{D.~Praetorius}
\address{Institute for Analysis and Scientific Computing,
       Vienna University of Technology,
       Wiedner Hauptstra\ss{}e 8-10,
       A-1040 Wien, Austria}
       \email{\{Michael.Feischl,\,Thomas.Fuehrer\}@tuwien.ac.at}
       \email{\{Melenk,\,Dirk.Praetorius\}@tuwien.ac.at}
\email{Michael.Karkulik@tuwien.ac.at\quad\rm(corresponding author)}

%%%%%%%%%%%%%%%%%%%%%%%%%%%%%%%%%%%%%%%%%%%%%%%%%%%%%%%%%%%%%%%%%%%%%

\begin{abstract}
We prove inverse-type estimates for the four classical boundary integral operators 
associated with the Laplace operator. These estimates are used to show 
convergence of an $h$-adaptive algorithm for the coupling of a finite element method
with a boundary element method which is driven by a weighted residual
error estimator. 
\end{abstract}

%%%%%%%%%%%%%%%%%%%%%%%%%%%%%%%%%%%%%%%%%%%%%%%%%%%%%%%%%%%%%%%%%%%%%

\maketitle

%====================================================================
\section{Introduction}
%====================================================================
\label{section:introduction}%

\noindent
{\sl A~posteriori} error estimation and adaptivity have a long history 
in finite element methods (FEMs), which harks back at least to 
the late 1970s. 
While the early mathematical analysis 
(see, e.g., the monographs 
\cite{ainsworth-oden00,babuska-strouboulis01,verfuerth96})
focussed on {\sl a posteriori} 
error estimation, significant progress has been made in the last decade 
in the analysis of the adaptive FEM (AFEM) regarding 
provable convergence and achievable convergence rates.
For linear elliptic model problems and discretizations
by fixed order polynomials on shape-regular meshes, convergence and even 
quasi-optimal convergence rates for AFEM  have been proved; we refer 
to \cite{ckns}
for symmetric problems as well as to~\cite{cn} for nonsymmetric
problems and the references therein. 

The situation is less developed for the adaptive boundary element method (ABEM)
and even worse for the adaptive coupling of FEM and BEM. For the ABEM based
on first kind integral equations, \cite{fkmp} and \cite{gantumur} proved very
recently convergence and optimality.  
Specifically, \cite{fkmp,gantumur} studied lowest order discretizations of equations 
related to the simple-layer operator and 
hypersingular operator. We highlight that the symmetry of these two operators 
is an important ingredient in the optimality proofs for the ABEM in 
\cite{fkmp,gantumur}. 
As a first step towards a full analysis of the more 
complex case of the adaptive coupling of FEM and BEM, we show in the 
present paper convergence for Costabel's symmetric coupling. 

Broadly speaking, the procedure in \cite{fkmp,gantumur} and the present work
relies on a framework delineated for AFEM in \cite{ckns}. 
The starting point for AFEM are 
reliable residual type error estimators. Ideally, the residual is 
measured in a dual norm; for example, in the classical Laplace Dirichlet 
problem with numerical approximation $u_h$ one has to evaluate 
$\|f + \Delta u_h\|_{H^{-1}(\Omega)}$. Since such duals norms are difficult 
to realize computationally, the classical residual error estimators mimic local versions of them 
by weighted $L^2$-norms of various components of the residual. 
The appropriate weight is given in terms of the local mesh size function $h$. 
Returning to the example of 
the Laplace Dirichlet problem, these are the elementwise
volume residuals  
$\|h (f + \Delta u_h)\|_{L^2(T)}$ and the edge/face jumps 
of the normal derivative 
$\|h^{1/2} [\partial_n u_h]\|_{L^2(E)}$.  
In effect, the residual is measured in stronger, but $h$-weighted 
Sobolev norms. Inverse estimates are therefore a key ingredient to showing 
efficiency of these estimators and convergence in the context of adaptive methods. 
A second feature of error estimators in the FEM is the local character 
of the volume and edge/face residuals. This feature stems from the fact that a {\em differential}
equation is considered. As a result, the only inverse estimates needed
in the FEM are classical ones relating stronger integer order Sobolev norms
to weaker ones. 

Our point of departure for ABEM are 
computable weighted residual error indicators made available in 
\cite{cc1997,cms,cmps,cs}. Analogously to the FEM, the nonlocal
nature of the norm in which to measure the residual (these are typically
non-integer Sobolev norms) is accounted for by $h$-weighted 
integer order Sobolev norms; for example, 
$\|h^{1/2} \nabla_\Gamma (\cdot)\|_{L^2(\Gamma)}$
is taken as a proxy for $|\cdot|_{H^{1/2}(\Gamma)}$, where $\nabla_\Gamma$ is the 
surface gradient on the surface $\Gamma$. 
The second feature of the residual error estimators in FEM mentioned
above is the local character of the residual. This feature is not present in 
boundary integral equations. For example, the weighted residual error estimator
for Symm's integral equation involving the simple-layer operator $V$ 
is $\|h^{1/2} \nabla_\Gamma (f - V u_h)\|_{L^2(\Gamma)}$;  
even when restricting the integral to a single element, 
the nonlocal nature of $V$ involves the numerical approximation $u_h$ on whole
surface $\Gamma$. 
As a result, the classical inverse estimates for spaces 
of piecewise polynomials, which were suitable in the FEM, are insufficient
for the BEM. The appropriate inverse estimates are provided in the present paper. 
In this connection, the works
\cite{fkmp,gantumur} are particularly relevant. Using similar techniques,
\cite{fkmp} considers the special case of the lowest order discretization 
of the simple-layer operator $\slp$, whereas the present article covers
arbitrary (fixed) order conforming discretizations of all four operators. 
The work of \cite{gantumur} leads to very similar results; possibly due to 
the use of wavelet techniques, \cite{gantumur} requires $C^{1,1}$-surfaces.
Instead, our analysis relies on techniques from local elliptic
regularity theory, and we may thus admit polyhedral surfaces here.
We finally note that 
all four operators appear in our formulation of the FEM-BEM coupling. 

The remainder of this work is organized as follows:
Section~\ref{section:invest} collects all necessary notations and preliminaries
(Section~\ref{section:invest:preliminaries}--\ref{section:invest:sobolev})
and proves the new inverse estimates
(Theorem~\ref{thm:invest} and Corollary~\ref{cor:invest}), which are the
main results of this work. Our analysis relies on elliptic regularity
estimates for the simple-layer potential $\widetilde\slp$ (Section~\ref{section:invest:aux})
and the double-layer potential $\widetilde\dlp$ (Section~\ref{sec:aux-dlp}).
In Section~\ref{section:afb}, we consider an adaptive algorithm
for Costabel's symmetric FEM-BEM coupling,
which is applied to
a linear transmission problem (Section~\ref{section:continuous}). Our
discretization includes the approximation of the given data so that an
implementation has to deal with discrete boundary integral operators
only (Section~\ref{section:perturbed}). 
We therefore extend the reliable error estimator
of~\cite{cs1995} to include data approximation terms in
Proposition~\ref{prop:aposteriori}.
Adapting the concept of estimator reduction~\cite{estconv}, which has
also been used for $(h-h/2)$-type error estimators in~\cite{afp},
we prove that the usual adaptive coupling (Algorithm~\ref{algorithm}) leads
to a perturbed contraction for the error estimator $\varrho_\ell$ and thus
obtain convergence $\varrho_\ell\to0$ as $\ell\to\infty$.
Since $\varrho_\ell$ provides an upper bound for the Galerkin error, which unlike~\cite{afp}
does \emph{not} rely on any saturation assumption, we thus obtain convergence of the
adaptive FEM-BEM coupling (Theorem~\ref{thm:convergence}).
A short Section~\ref{section:fembem:extension} discusses the extension of these 
result to nonlinear transmission problems. 
Numerical experiments in Section~\ref{section:numerics} illustrate the convergence of the 
adaptive FEM-BEM coupling procedure and give empirical evidence that the optimal order
of convergence is, in fact, achieved.

%====================================================================
\section{Inverse estimates for integral operators}
%====================================================================
\label{section:invest}%
\def\n{ {\mathbb{n}}}
\def\L{ {\mathcal{L}}}
\def\gint{ {\gamma_1^{\rm int}}}
\def\gext{ {\gamma_1^{\rm ext}}}
\def\supp{{\rm supp}}
\def\trace{ {\gamma_0^{\rm int}}}
\def\trext{ {\gamma_0^{\rm ext}}}
\def\el{ {E} }
  \def\S{\mathbb{S}}

\noindent
In this section, we prove certain inverse estimates for the 
four classical boundary integral operators associated with the Laplace 
operator. Independently of our work, similar estimates for \revision{fixed} order 
\revision{piecewise polynomials} have been proved in~\cite{gantumur} by means of 
wavelet-based techniques. For technical reasons, \cite{gantumur} 
assumes the boundary to be fairly smooth, namely, $C^{1,1}$.
In contrast, the present analysis is based on PDE techniques and 
allows us to treat polygonal/polyhedral boundaries. 
We mention that the estimate~\eqref{eq:invest:V} below has already been 
shown in our earlier work~\cite{fkmp} in a discrete setting for 
lowest order elements $\Psi_\ell\in\PP^0(\mesh_\ell)$ and is generalized 
here to the case of arbitrary $\psi\in L^2(\Gamma)$.

%-------------------------------------------------------------------------
\subsection{Preliminaries \& general assumptions}
%-------------------------------------------------------------------------
\label{section:invest:preliminaries}
Let $\Omega$ be a bounded Lipschitz domain in $\R^d$, for $d\ge2$, 
with polygonal/polyhedral Lipschitz boundary $\Gamma=\partial\Omega$. 
The exterior
unit normal vector field is denoted by $\normal$. We define $\Omega^\e:=\R^d\setminus\overline\Omega$. Let $\mesh_\ell$ denote a conforming triangulation of $\Gamma$ into simplices, i.e.
\begin{itemize}
\item $\mesh_\ell$ is a finite set of non-degenerate $(d-1)$-dimensional compact surface simplices, i.e., affine images of the reference simplex $E_{\rm ref}\subset\R^{d-1}$ with positive surface measure;
\item $\Gamma = \bigcup_{E\in\mesh_\ell} E$, i.e., $\Gamma$ is covered by $\mesh_\ell$;
\item for each $E,E'\in\mesh_\ell$ with $E\neq E'$, the intersection $E\cap E'$ is either empty,
  or a $j$-dimensional simplex for $j=0, \dots, d-2$,
  i.e., a joint node, or a joint edge, or a joint face, etc.
\end{itemize}
Moreover, we assume that $\mesh_\ell$ is $\kappa$-shape regular, i.e., 
\begin{align}\label{def:kappa:3d}
 \max_{E\in\mesh_\ell}\frac{\diam(E)^{d-1}}{|E|} \le \kappa < \infty
 \quad\text{for }d\ge3
\end{align}
with $|\cdot|$ the $(d-1)$-dimensional surface measure and $\diam(\cdot)$ the Euclidean diameter, whereas
\begin{align}\label{def:kappa:2d}
 \max_{\substack{E,E'\in\mesh_\ell\\E\cap E'\neq\emptyset}}\frac{\diam(E)}{\diam(E')}\le \kappa < \infty
 \quad\text{for }d=2.
\end{align}
Note that, for $d\ge3$, conformity of $\mesh_\ell$ and $\kappa$-shape regularity~\eqref{def:kappa:3d} also imply~\eqref{def:kappa:2d} (though with a different, but bounded constant $\widetilde\kappa$).
With each triangulation $\mesh_\ell$, we associate the local mesh size function $h_\ell\in L^\infty(\Gamma)$ which is defined elementwise by $h_\ell|_E := h_\ell(E) := |E|^{1/(d-1)}$ for all $E\in\mesh_\ell$.
We stress that $\kappa$-shape regularity of $\mesh_\ell$ implies $h_\ell|_E \simeq \diam(E)$.

Let $\gamma \subseteq \Gamma$ be a relatively open subset of $\Gamma$.
With
\begin{align}
 \mesh_\ell^\gamma := \mesh_\ell|_\gamma
 := \set{E\in\mesh_\ell}{E\subseteq\overline\gamma},
\end{align}
we denote the restriction of $\mesh_\ell$ to $\gamma$. It is always assumed 
that $\gamma$ is resolved by $\mesh_\ell$, i.e., $\mesh_\ell^\gamma$ is
a $\kappa$-shape regular and conforming triangulation of $\gamma$.

%-------------------------------------------------------------------------
\subsection{Sobolev spaces and boundary integral operators}
%-------------------------------------------------------------------------
\label{section:invest:sobolev}
In this section, we very briefly fix our notation concerning Sobolev spaces and 
boundary integral operators and refer the reader to 
the monographs~\cite{mclean,hsiaowendland,ss,verchota} for 
further details and the precise definitions. 

For the boundary $\Gamma = \partial\Omega$ of the bounded Lipschitz domain $\Omega$, we denote by 
$\nabla_\Gamma(\cdot)$ the surface gradient. The Sobolev space $H^1(\Gamma)$ can be defined as 
the completion of the Lipschitz continuous functions on $\Gamma$ with respect to the norm
\begin{align*}
  \norm{u}{H^1(\Gamma)}^2 := \norm{u}{L^2(\Gamma)}^2 + \norm{\nabla_\Gamma u}{L^2(\Gamma)}^2.
\end{align*}
We denote by $\trace(\cdot)$ the interior trace operator, i.e., the $\trace u$ is the 
restriction of a function $u\in H^1(\Omega)$ to the boundary $\Gamma$. 
The space $H^{1/2}(\Gamma)$ is the trace space of $H^1(\Omega)$ equipped in the standard
way with the quotient norm.
For relatively open subsets $\gamma \subset\Gamma$ and $s \in \{-1/2,0,1/2\}$, we denote by
\begin{align*}
  H^{1/2+s}(\gamma) &= \set{v|_\gamma}{v\in H^{1/2+s}(\Gamma)},\\
  \H^{1/2+s}(\gamma)&= \set{v|_\gamma}{v \in H^{1/2+s}(\Gamma), \operatorname*{supp} v \subset \overline{\gamma}}. 
\end{align*}
the space of all restrictions of functions to $\gamma$ and endow this spaces with the corresponding quotient norms.
In particular, if $v\in\H^{1/2+s}(\gamma)$ is extended by zero to 
the entire boundary $\Gamma$, then $v\in H^{1/2+s}(\Gamma)$ and $\norm{v}{\H^{1/2+s}(\gamma)} = \norm{v}{H^{1/2+s}(\Gamma)}$.
Finally, negative order spaces
\begin{align*}
  H^{-1/2}(\Gamma) &:= H^{1/2}(\Gamma)',\quad
  \H^{-(1/2+s)}(\gamma) &:= H^{1/2+s}(\gamma)',\quad\text{and}\quad
  H^{-(1/2+s)}(\gamma) &:= \H^{1/2+s}(\gamma)'
\end{align*}
are defined by duality, where duality pairings $\langle \cdot ,\cdot\rangle$ are understood to extend the 
standard $L^2$-scalar product. 
We note the continuous inclusions
\begin{align*}
  \H^{\pm(1/2+s)}(\gamma) &\subseteq H^{\pm(1/2+s)}(\gamma) 
 \quad\text{as well as}\quad
  \H^{\pm(1/2+s)}(\Gamma) = H^{\pm(1/2+s)}(\Gamma).
\end{align*}
The interior conormal derivative operator 
$\gamma_1^\interior:H^1_\Delta(\Omega)\rightarrow H^{-1/2}(\Gamma)$, where
\begin{align*}
  H^1_\Delta(\Omega):=
  \{u \in H^1(\Omega)\,|\, -\Delta u \in L^2(\Omega)\},
\end{align*}
is defined by the 
first Green's formula, viz., 
\begin{align}
  \label{eq:first Green identity}
  \dual{\gamma_1^\interior u}{v}_\Gamma
  =\dual{\nabla u}{\nabla v}_\Omega
  - \dual{-\Delta u}{v}_\Omega
  \quad\text{for all }v\in H^1(\Omega).
\end{align}

\begin{remark}
\label{rem:gamma_1}
The operator $\gamma_1^\interior$ generalizes the classical normal derivative 
operator: if $u \in H^1_\Delta(\Omega)$ is sufficiently smooth near a boundary
point $x_0$, then  $\gamma_1^\interior$ can be represented near $x_0$ by a function 
given by the pointwise defined  normal derivative $\partial_{\normal} u$.  
\end{remark}

The exterior trace $\trext$ and the exterior conormal derivative
operator $\gext$ are defined analogously to their interior counterparts. 
To that end, we fix a bounded Lipschitz domain $U \subset {\mathbb R}^d$ 
with $\overline{\Omega} \subset U$. The exterior trace operator 
$\trext:H^1(U\setminus\overline{\Omega}) \rightarrow H^{1/2}(\Gamma)$ 
is defined by restricting to $\Gamma$, and the exterior conormal derivative 
$\gext$ is characterized by 
 $ \dual{\gext u}{v}_\Gamma
  =\dual{\nabla u}{\nabla v}_{U\setminus\overline{\Omega}}
  - \dual{-\Delta u}{v}_{U\setminus\overline{\Omega}}$
  for all $v\in H^1(U\setminus\overline{\Omega}).$

For a function $u$ that admits both derivatives, we define
the jumps
$[\gamma_1 u] := \gext u - \gint u$, as well as $[u] = \trext u - \trace u$.

We denote by $G$ the fundamental solution of the $d$-dimensional Laplacian
\begin{align}
 G(x,y) = \begin{cases}
 -\frac{1}{|\S^1|}\,\log|x-y|,\quad&\text{for }d=2,\\
 +\frac{1}{|\S^{d-1}|}\,|x-y|^{-(d-2)},&\text{for }d\ge3,
 \end{cases}
\end{align}
where $|\S^{d-1}|$ denotes the surface measure of the Euclidean sphere in $\R^d$, e.g., $|\S^1|=2\pi$ and $|\S^2|=4\pi$.
The classical simple-layer potential $\widetilde\slp$ and the double layer potential
$\widetilde \dlp$ are defined by 
\begin{equation*}
(\widetilde \slp \psi)(x):= \int_\Gamma G(x,y) \psi(y)\,d\Gamma(y), 
\qquad 
(\widetilde \dlp v)(x):= \int_\Gamma \partial_{\normal(y)} G(x,y) \psi(y)\,d\Gamma(y), 
\qquad x \in {\mathbb R}^d\setminus \Gamma; 
\end{equation*}
here, $\partial_{\normal(y)}$ denotes the (outer) normal derivative with respect to the variable $y$. 
These pointwise defined operators can be extended to bounded linear operator with 
\begin{align}
\label{eq:mapping-properties-potentials}
 \widetilde\slp\in L\big(H^{-1/2}(\Gamma);H^1(U)\big)
 \quad\text{and}\quad
 \widetilde\dlp\in L\big(H^{1/2}(\Gamma);H^1(U\setminus\Gamma)\big).
\end{align}
It is classical that 
$\Delta\widetilde\slp\psi = 0 = \Delta\widetilde\dlp v$ in $U\setminus\Gamma$ 
for all $\psi\in H^{-1/2}(\Gamma)$ and $v\in H^{1/2}(\Gamma)$. 
The simple-layer, double-layer, adjoint double-layer, and the hypersingular operators are 
defined as follows: 
\begin{align}
  \slp = \trace\widetilde\slp, \quad \dlp = \frac{1}{2}+\trace\widetilde\dlp, \quad \dlp'=-\frac{1}{2}+\gint\widetilde\slp, \text{ and } \quad \hyp = -\gint\widetilde\dlp.
\end{align}
As is shown in
\revision{\cite{hsiaowendland,mclean,verchota}},  these linear operators 
have the following mapping properties for $s \in \{-1/2,0,1/2\}$ and representations 
(in the case of the hypersingular operator $\hyp$, the 
integral is understood as a \revision{part finite integral}): 
\begin{align}
 \label{def:slp}
 \slp&\in L(\H^{-1/2+s}(\gamma);H^{1/2+s}(\gamma)),
 &(\slp\psi)(x) &= \int_\Gamma G(x,y)\,\psi(y)\,d\Gamma(y),\\
 \label{def:dlp}
 \dlp&\in L(\H^{1/2+s}(\gamma);H^{1/2+s}(\gamma)),
 &(\dlp v)(x) &= \int_\Gamma \partial_{\normal(y)}G(x,y)\,v(y)\,d\Gamma(y),\\
 \label{def:adlp}
 \dlp'&\in L(\H^{-1/2+s}(\gamma);H^{-1/2+s}(\gamma)),
 &(\dlp'\psi)(x) &= \int_\Gamma \partial_{\normal(x)}G(x,y)\,v(y)\,d\Gamma(y),\\
 \label{def:hyp}
 \hyp&\in L(\H^{1/2+s}(\gamma);H^{-1/2+s}(\gamma)),
 &(\hyp v)(x) &= -\partial_{\normal(x)}\int_\Gamma \partial_{\normal(y)}G(x,y)\,v(y)\,d\Gamma(y).
\end{align}

%-------------------------------------------------------------------------
\subsection{Statement of main result on inverse estimates}
%-------------------------------------------------------------------------
The following theorem is the main result of this work
and the mathematical core of the arguments that allow to 
transfer convergence results from AFEM to ABEM.

\begin{theorem}
\label{thm:invest}
There exists
a constant $\setc{inv}>0$ such that the following estimates hold:
\begin{align}
\label{eq:invest:V}
\norm{h_\ell^{1/2}\nabla_\Gamma\slp\psi}{L^2(\gamma)}
& \le\c{inv}\big(\norm{\psi}{\H^{-1/2}(\gamma)}
+ \norm{h_\ell^{1/2}\psi}{L^2(\gamma)}\big),
\\
\label{eq:invest:Kadj}
\norm{h_\ell^{1/2}\dlp'\psi}{L^2(\gamma)}
& \le\c{inv}\big(\norm{\psi}{\H^{-1/2}(\gamma)}
+ \norm{h_\ell^{1/2}\psi}{L^2(\gamma)}\big),
\\
\label{eq:invest:K}
\norm{h_\ell^{1/2}\nabla_\Gamma\dlp v}{L^2(\gamma)}
&
\le\c{inv}\big(\norm{v}{\H^{1/2}(\gamma)}
+ \norm{h_\ell^{1/2}\nabla_\Gamma v}{L^2(\gamma)}\big),
\\
\label{eq:invest:W}
\norm{h_\ell^{1/2}\hyp v}{L^2(\gamma)}
&\le\c{inv}\big(\norm{v}{\H^{1/2}(\gamma)}
+ \norm{h_\ell^{1/2}\nabla_\Gamma v}{L^2(\gamma)}\big),
\end{align}
for all functions $\psi \in L^2(\gamma)$
and all $v \in \widetilde H^1(\gamma)$.
The constant $\c{inv}>0$ depends only on $\Gamma$, $\gamma$, and 
$\kappa$-shape regularity of $\mesh_\ell$.
\end{theorem}

\begin{remark}
The estimates~\eqref{eq:invest:V}--\eqref{eq:invest:W} are rather easy to show for globally 
quasi-uniform meshes, i.e., $h_\ell(E) \simeq h_\ell(E')$ for all $E$, $E'\in\mesh_\ell$. 
For example, to see~\eqref{eq:invest:V}, one recalls stability of $\slp:L^2(\gamma)\to H^1(\gamma)$ 
to get 
\begin{align}\label{eq:invest:uniform}
 \norm{h_\ell^{1/2}\nabla_\Gamma\slp\psi}{L^2(\gamma)}
 \lesssim h_\ell^{1/2} \norm{\slp\psi}{H^1(\gamma)}
 \lesssim h_\ell^{1/2} \norm{\psi}{L^2(\gamma)}
 \simeq \norm{h_\ell^{1/2}\psi}{L^2(\gamma)}.
\end{align}
\revision{
One sees that~\eqref{eq:invest:uniform} is slightly stronger than~\eqref{eq:invest:V}, where an additional term $\norm{\psi}{\H^{-1/2}(\gamma)}$ arises on the right-hand side.
}
\end{remark}%

For each element $E\in\mesh_\ell^\gamma$, let $\gamma_E:E_{\rm ref}\to E$ denote the affine bijection from the reference simplex $E_{\rm ref}\subset\R^{d-1}$ onto $E$. We introduce the space of (discontinuous) piecewise polynomials of degree $q$ and the spaces of continuous piecewise polynomials of degree $p$ by 
\begin{align}
 &\PP^q(\mesh_\ell^\gamma)
 := \set{\Psi_\ell\in L^2(\gamma)}{\forall E\in\mesh_\ell^\gamma\quad
 \Psi_\ell\circ\gamma_E\text{ is a polynomial of degree }\le q},\\
 &\SS^p(\mesh_\ell^\gamma) := \PP^p(\mesh_\ell^\gamma)\cap C(\gamma),
 \quad\text{and}\quad
 \SS^p_0(\mesh_\ell^\gamma) := \set{ V_\ell|_\gamma}{V_\ell\in\SS^p(\mesh_\ell) \text{ with } \supp(V_\ell)\subseteq\overline\gamma}. 
\end{align}
We note the inclusions
$\PP^q(\mesh_\ell^\gamma)\subset L^2(\gamma) \subset \H^{-1/2}(\gamma)$, $\SS^p_0(\mesh_\ell^\gamma)
\subset \H^1(\gamma) \subset \H^{1/2}(\gamma)$, and $\SS^p(\mesh_\ell^\gamma)\subseteq H^1(\gamma)$, as well as $\SS^p_0(\mesh_\ell) = \SS^p(\mesh_\ell)$ in case of $\gamma=\Gamma$. If we now restrict the estimates~\eqref{eq:invest:V}--\eqref{eq:invest:W} of Theorem~\ref{thm:invest} to discrete functions $\Psi_\ell\in\PP^q(\mesh_\ell^\gamma)$ and $V_\ell\in\SS^p_0(\mesh_\ell^\gamma)$,
we obtain the following estimates.

\begin{corollary}
\label{cor:invest}
There exists
a constant $\setc{invtilde}>0$ such that the following estimates hold:
  \begin{align}
   \label{eq:cor:invest:V}
   \max\{1,q\}^{-1}\norm{h_\ell^{1/2}\nabla_\Gamma\slp\Psi_\ell}{L^2(\gamma)}
   &\le\c{invtilde}\norm{\Psi_\ell}{\H^{-1/2}(\gamma)},\\
   \label{eq:cor:invest:Kadj}
   \max\{1,q\}^{-1}\norm{h_\ell^{1/2}\dlp'\Psi_\ell}{L^2(\gamma)}
   &\le\c{invtilde}\norm{\Psi_\ell}{\H^{-1/2}(\gamma)},\\
   \label{eq:cor:invest:K}
   p^{-1}\norm{h_\ell^{1/2}\nabla_\Gamma\dlp V_\ell}{L^2(\gamma)}
   &\le\c{invtilde}\norm{V_\ell}{\H^{1/2}(\gamma)},\\
   \label{eq:cor:invest:W}
   p^{-1}\norm{h_\ell^{1/2}\hyp V_\ell}{L^2(\gamma)}
   &\le\c{invtilde}\norm{V_\ell}{\H^{1/2}(\gamma)},
  \end{align}
  for all discrete functions $\Psi_\ell\in\PP^q(\mesh_\ell^\gamma)$ and
$V_\ell\in\SS^p_0(\mesh_\ell^\gamma)$.
  The constant $\c{invtilde}>0$ depends only on $\gamma$ and the
  shape regularity constant $\kappa$ of $\mesh_\ell$, but is independent of the polynomial degrees $q\ge0$ resp.\ $p\ge1$.
\end{corollary}
\begin{remark}
  We stress that for discrete spaces with locally varying polynomial degrees, e.g.
  \begin{align*}
    \PP^{q_\ell}(\mesh_\ell^\gamma)
     := \set{\Psi_\ell\in L^2(\gamma)}{\forall E\in\mesh_\ell^\gamma\quad
     \Psi_\ell\circ\gamma_E\text{ is a polynomial of degree }\le q_\ell(\el)},
  \end{align*}
  the inverse estimates of Corollary~\ref{cor:invest} remain true as long as
  the polynomial degrees are comparable on neighboring elements, i.e.,
  $q_\ell(E) \simeq q_\ell(E')$ for \mbox{$E\cap E' \neq \emptyset$.} For example, the inverse estimate~\eqref{eq:cor:invest:V}
  involving $\slp$ then reads
  \begin{align*}
    \norm{h_\ell^{1/2}\max\{1,q_\ell\}^{-1}\nabla_\Gamma\slp\Psi_\ell}{L^2(\gamma)}
    &\le\c{invtilde}\norm{\Psi_\ell}{\H^{-1/2}(\gamma)} \quad\text{ for all } \Psi_\ell \in \PP^{q_\ell}(\mesh_\ell^\gamma),
  \end{align*}
  and the estimates~\eqref{eq:cor:invest:Kadj}--\eqref{eq:cor:invest:W} can be extended analogously.
  We refer the reader to~\cite[Theorem 4.4]{k}, where the inverse estimates are proved in this extended fashion.
\end{remark}
\begin{proof}[Proof of Corollary \ref{cor:invest}]
  Our starting point are two inverse estimates from~\cite[Theorem 3.9]{georgoulis} and~\cite[Proposition~3]{akp}:
  \begin{align}
    \label{eq:foo-1}
    \max\{1,q\}^{-1}\norm{h_\ell^{1/2}\Psi_\ell}{L^2(\gamma)}
    \lesssim \norm{\Psi_\ell}{\H^{-1/2}(\gamma)}
    \quad\text{for all }\Psi_\ell\in\PP^q(\mesh_\ell^\gamma),\\
    \label{eq:foo-2}
    p^{-1}\norm{h_\ell^{1/2}\nabla_\Gamma V_\ell}{L^2(\gamma)}
    \lesssim \norm{V_\ell}{\H^{1/2}(\gamma)}
    \quad\text{for all }V_\ell\in\SS^p_0(\mesh_\ell^\gamma),
  \end{align}
  where the \revision{hidden} constants depend solely on $\Gamma$ and the $\kappa$-shape regularity
  of $\mesh_\ell^\gamma$. 
  Combining \eqref{eq:foo-1} with~\eqref{eq:invest:V}--\eqref{eq:invest:Kadj} 
  leads to \eqref{eq:cor:invest:V}--\eqref{eq:cor:invest:Kadj};
  the bound \eqref{eq:foo-2} in conjunction 
  with~\eqref{eq:invest:K}--\eqref{eq:invest:W} yields 
  \eqref{eq:cor:invest:K}--\eqref{eq:cor:invest:W}.
\end{proof}

The remainder of this section is devoted to the proof of Theorem~\ref{thm:invest}. 
The proof will first be given for $\gamma=\Gamma$, and the general case 
$\gamma\subsetneqq\Gamma$ is deduced from it afterwards. 

On the technical side, an important difficulty of the proof of Theorem~\ref{thm:invest}
arises from the fact that the boundary integral operators are nonlocal. 
We cope with this issue by splitting the operators into near field 
and far field contributions, each requiring different tools. The analysis
of the near field part relies on local arguments and stability properties
of the BIO. For the far field part, the key observation
is that the BIOs are derived from two potentials, namely, the simple-layer 
potential $\widetilde \slp$ and the double-layer potential $\widetilde \dlp$ 
by taking appropriate traces. Since these potentials solve elliptic equations, 
inverse type estimates (``Caccioppoli inequalities'') are available for them. 
The study of these two potentials is the topic of Sections~\ref{section:invest:aux}
and~\ref{sec:aux-dlp}. We will discuss the case
of the simple-layer potential $\widetilde \slp$ in greater detail first
and be briefer afterwards in our treatment of the 
double-layer potential $\widetilde \dlp$, since the basic arguments are  
similar to  those for $\widetilde \slp$. 
%-------------------------------------------------------------------------
\subsection{Far field and near field estimates for the simple layer potential}
\label{section:invest:aux}
%-------------------------------------------------------------------------
\revision{We start with subsection~\ref{sec:notation-slp}, where we introduce the decomposition of the simple layer potential
into far field and near field. For either of this parts, we provide inverse estimates. Subsection~\ref{sec:nearfield-slp} is
devoted to the derivation of inverse estimates for the near field parts, whereas we deal with far field parts in
subsection~\ref{sec:farfield-slp}}

%-------------------------------------------------------------------------
\subsubsection{Notation and decomposition into near field and far field}
\label{sec:notation-slp}
%-------------------------------------------------------------------------
\def\x{{\bf x}}
For each element $E\in\mesh_\ell$ and $\delta>0$, we define the neighborhood
$U_E$ of $E$ by 
\begin{align}
 E \subset U_E := \bigcup_{x\in E} B_{2\delta h_\ell(E)}(x),
\end{align}
where $B_{\eps}(x) := \left\{ y\in\R^d \mid \abs{x-y}<\eps \right\}\subset\R^d$.
By $\kappa$-shape regularity of $\mesh_\ell$,
there exist $\delta>0$ and $M\in\N$ such that
$\Gamma\cap U_E$ is contained in the patch $\omega_\ell(E)$ of $E$, i.e., 
\begin{align}
\label{UE:patch}
  \Gamma\cap U_E \subseteq \omega_\ell(E):=\bigcup\set{E^\prime\in\mesh_\ell}{E^\prime\cap E\neq\emptyset},
\end{align}
and that the covering $\Gamma\subset\bigcup_{E\in\mesh_\ell}U_E$ is locally finite, i.e., 
\begin{align}\label{UE:overlap}
 \#\set{U_E}{E\in\mesh_\ell\text{ and }x\in U_E}\le M
 \quad\text{for all }x\in\R^d \text{ and $\ell\in\N$.}
\end{align}
\revision{%
Finally, we fix a bounded domain $U\subset\R^d$ such that
\begin{align}\label{def:U}
U_E \subset U
\quad\text{for all }E\in\mesh_\ell.
\end{align}%
}%
To deal with the nonlocality of the integral operators, we define
for functions $\psi\in L^2(\Gamma)$ and $E\in\mesh_\ell$ the near field and the far field
of the simple-layer potential $u_\slp = \widetilde\slp\psi$ by
\def\near{{\rm near}}
\def\far{{\rm far}}
\begin{align}
\label{eq:uslp_far_uslp_near}
 u_{\slp,E}^\near := \widetilde\slp(\psi\chi_{\Gamma\cap U_E})
 \quad\text{and}\quad
 u_{\slp,E}^\far := \widetilde\slp(\psi\chi_{\Gamma\setminus U_E}),
\end{align}
where $\chi_\omega$ denotes the characteristic function of the set $\omega\subset\R^d$.
We have the obvious identity
\begin{align}
 u_\slp =u_{\slp,E}^\near + u_{\slp,E}^\far
 \quad\text{for all }E\in\mesh_\ell.
\end{align}
In our analysis, we will treat $u_{\slp,E}^\near$ and $u_{\slp,E}^\far$ separately, starting
with the simpler case of $u_{\slp,E}^\near$. 
%----------------------------------------------------------------------------------------
\subsubsection{Inverse estimates for the near field part $u_{\slp,E}^\near$}
\label{sec:nearfield-slp}
%----------------------------------------------------------------------------------------
\revision{The near field parts of a potential can be treated with local arguments and the stability properties of the associated
boundary integral operators.}
\begin{lemma}\label{lemma:nearfield:V:1}
There exists a constant $\setc{xxx} > 0$ depending only on $\Gamma$
and the $\kappa$-shape regularity of $\mesh_\ell$ such that for arbitrary 
$E \in \mesh_\ell$ and $\Psi_\ell^E \in \PP^0(\mesh_\ell)$
with $\supp\left( \Psi_\ell^E \right)\subseteq\omega_\ell(E)$ there holds 
\begin{align*}
  \sum_{E\in\mesh_\ell}\norm{\nabla\widetilde\slp\Psi_\ell^E}{L^2(U_E)}^2
  \leq \c{xxx} \sum_{E\in\mesh_\ell}\norm{h_\ell^{1/2}\Psi_\ell^E}{L^2(\omega_\ell(E))}^2.
\end{align*}
\end{lemma}

\begin{proof}
  We fix an element $E\in\mesh_\ell$. Denoting by 
  $\Psi_\ell^E(E')$ the value of $\Psi_\ell^E\in\PP^0(\mesh_\ell)$ on the element $E'$
  we compute 
  \begin{align*}
    (\nabla \widetilde\slp\Psi_\ell^E)(x)
= \sum_{E'\in\omega_\ell(E)}\Psi_\ell^E(E')\int_{E'}\nabla_xG(x,y)\,d\Gamma( y)
  \quad\text{for all } x \in \R^d\setminus \Gamma.
  \end{align*}
  The number of elements $E'$ in the patch $\omega_\ell(E)$ is bounded in terms
of the shape regularity constant $\kappa$ and thus
  \begin{align*}
    \abs{(\nabla \widetilde\slp\Psi_\ell^E)(x)}^2
\lesssim
\sum_{E'\in\omega_\ell(E)}
\abs{\Psi_\ell^E(E')}^2 \Big(\int_{E'}\big|\nabla_ xG(x,y)\big|\,d\Gamma( y)\Big)^2
  \end{align*}
with some $E$-independent constant, which depends only on $\kappa$.
  Since the mesh $\mesh_\ell$ is $\kappa$-shape regular, we can select
  a constant $c>0$, which depends solely on $\kappa$, such that
  $U_E \subseteq B_{ch_\ell(E)}(b_{E^\prime})$ and
  $E^\prime \subseteq B_{ch_\ell(E)}(b_{E^\prime})$
for each neighbor $E^\prime\in\omega_\ell(E)$. 
  We integrate over $U_E$ and estimate the remaining integral
  \begin{align*}
    \int_{U_E}\abs{(\nabla \widetilde\slp\Psi_\ell^E)(x)}^2\,d x
&\lesssim
\sum_{E'\in\omega_\ell(E)}\abs{\Psi_\ell^E(E')}^2
\int_{U_E}\Big(\int_{E'}\big|\nabla_ xG(x,y)\big|\,d\Gamma( y)\Big)^2\,d x\\
&\lesssim
\sum_{E'\in\omega_\ell(E)}\abs{\Psi_\ell^E(E')}^2
\int_{B_{ch_\ell(E)}(b_{E'})}
\Big(\int_{B_{ch_\ell(E)}(b_{E'})\cap\Gamma_{E'}}\frac{1}{\abs{ x- y}^{d-1}}\,d\Gamma( y)\Big)^2\,d x,
  \end{align*}
  where $\Gamma_{E'}$ denotes the hyperplane that is spanned by $E'$.
Scaling arguments then yield
  \begin{align*}
\int_{U_E}\abs{(\nabla \widetilde\slp\Psi_\ell^E)(x)}^2\,d x
&\lesssim
\sum_{E'\in\omega_\ell(E)}\abs{\Psi_\ell^E(E')}^2 h_\ell(E)^d
\int_{B_1(0)}\Big(\int_{B_{1}(0)\cap\R^{d-1}}\frac{1}{\abs{ x- y}^{d-1}}\,d\Gamma( y)\Big)^2\,d x \\
  &\lesssim \sum_{E'\in\omega_\ell(E)}\abs{\Psi_\ell^E(E')}^2 h_\ell(E)^d
  \lesssim \norm{h_\ell^{1/2}\Psi_\ell^E}{L^2(\omega_\ell(E))}^2.
  \end{align*}
  Summing this last estimate over all $E\in\mesh_\ell$, we conclude the proof.
\end{proof}

%----------------------------------------------------------------------------------------------------------
\begin{proposition}[Near field bound for $\widetilde\slp$]
\label{prop:nearfield:V}
There is a 
constant $\setc{nearfield}>0$ depending only on $\Gamma$ and the $\kappa$-shape regularity of $\mesh_\ell$
such that the near field part $u_{\slp,E}^\near$ satisfies 
\revision{$u_{\slp,E}^\near \in H^1(U)$} and 
$\gamma_0^\interior u_{\slp,E}^\near \in H^1(\Gamma)$ as well as 
\begin{align}\label{eq1:nearfield:V}
  \sum_{E\in\mesh_\ell}\norm{h_\ell^{1/2}\nabla_\Gamma \gamma_0^\interior u_{\slp,E}^\near}{L^2(E)}^2
  +
  \sum_{E\in\mesh_\ell}\norm{\nabla u_{\slp,E}^\near}{L^2(U_E)}^2
  \le \c{nearfield}\,\norm{h_\ell^{1/2}\psi}{L^2(\Gamma)}^2.
\end{align}
\end{proposition}

\begin{proof}
  The stability assertion $\slp:L^2(\Gamma)\rightarrow H^1(\Gamma)$ proved in 
  \cite{verchota} gives, for each  $E \in \mesh_\ell$, 
  \begin{align*}
   \norm{\nabla_\Gamma \gamma_0^\interior u_{\slp,E}^\near}{L^2(E)}
   \leq \norm{\slp( \psi\chi_{U_E\cap\Gamma})}{H^1(\Gamma)}
   \lesssim \norm{\psi\chi_{U_E\cap\Gamma}}{L^2(\Gamma)}
   = \norm{\psi}{L^2(U_E\cap\Gamma)}.
  \end{align*}
  Summing the last estimate over all $E \in \mesh_\ell$ and using the finite overlap
  property~\eqref{UE:overlap} of the set $U_E$, we arrive at 
  \begin{align*}
   \sum_{E\in\mesh_\ell}\norm{h_\ell^{1/2}\nabla_\Gamma \gamma_0^\interior u_{\slp,E}^\near}{L^2(E)}^2
   &= \sum_{E\in\mesh_\ell}h_\ell(E)\norm{\nabla_\Gamma \gamma_0^\interior u_{\slp,E}^\near}{L^2(E)}^2\\
   &\lesssim \sum_{E\in\mesh_\ell}h_\ell(E)\norm{\psi}{L^2(U_E\cap\Gamma)}^2
   \simeq \norm{h_\ell^{1/2}\psi}{L^2(\Gamma)}^2,
  \end{align*}
  where all estimates depend only on the $\kappa$-shape regularity constant $\kappa$. This bounds the first term on the left-hand side of~\eqref{eq1:nearfield:V}.
  To bound the second term, let $\Pi_\ell$ denote the $L^2(\Gamma)$-orthogonal projection onto $\PP^0(\mesh_\ell)$.
  We decompose the near field as 
$u_{\slp,E}^\near
= \widetilde \slp(\Pi_\ell(\psi\chi_{\Gamma\cap U_E}))
+ \widetilde \slp\big( (1-\Pi_\ell)\psi\chi_{\Gamma\cap U_E}\big)$.
The condition $\supp(\psi\chi_{\Gamma\cap U_E})\subseteq\omega_\ell(E)$
implies $\supp\left(\Pi_\ell(\psi\chi_{\Gamma\cap U_E})\right) \subseteq \omega_\ell(E)$ and therefore, 
  taking $\Psi_\ell^E = \Pi_\ell(\psi\chi_{\Gamma\cap U_E})$ in Lemma~\ref{lemma:nearfield:V:1}, we conclude
  \begin{align}
  \label{eq:foo-10}
    \sum_{E\in\mesh_\ell}\norm{\nabla\widetilde\slp(\Pi_\ell(\psi\chi_{\Gamma\cap U_E}))}{L^2(U_E)}^2
    \lesssim \sum_{E\in\mesh_\ell}\norm{h_\ell^{1/2}\Pi_\ell(\psi\chi_{\Gamma\cap U_E})}{L^2(\omega_\ell(E))}^2
    \lesssim \norm{h_\ell^{1/2}\psi}{L^2(\Gamma)}^2,
  \end{align}
  where we used the local $L^2$-stability of $\Pi_\ell$ in the last estimate.
Recalling the stability $\widetilde\slp:H^{-1/2}(\Gamma)\to H^1(U)$ of \eqref{eq:mapping-properties-potentials}
we can estimate 
\begin{align*}
  \norm{\nabla \widetilde\slp\big((1-\Pi_\ell) \psi\chi_{\Gamma\cap U_E}\big)}{L^2(U_E)}
  &\lesssim \norm{\nabla \widetilde\slp\big((1-\Pi_\ell) \psi\chi_{\Gamma\cap U_E}\big)}{L^2(U)}\\
  &\lesssim \norm{(1-\Pi_\ell) \psi\chi_{\Gamma\cap U_E}}{H^{-1/2}(\Gamma)} 
\end{align*}
Together with a local approximation result for $\Pi_\ell$ from~\cite[Theorem~4.1]{ccdpr:symm}, we obtain
\begin{align}
\nonumber 
\sum_{E\in\mesh_\ell}\norm{\nabla \widetilde\slp\big((1-\Pi_\ell) \psi\chi_{\Gamma\cap U_E}\big)}{L^2(U_E)}^2
&
\lesssim \sum_{E\in\mesh_\ell}\norm{(1-\Pi_\ell) \psi\chi_{\Gamma\cap U_E}}{H^{-1/2}(\Gamma)}^2
\\ &
  \label{eq:foo-20}
\lesssim \sum_{E\in\mesh_\ell}\norm{h_\ell^{1/2}(\psi\chi_{\Gamma\cap U_E})}{L^2(\Gamma)}^2
\simeq \norm{h_\ell^{1/2}\psi}{L^2(\Gamma)}^2.
\end{align}
\revision{The combination of~\eqref{eq:foo-10}--\eqref{eq:foo-20}} yields 
 the desired estimate in~\eqref{eq1:nearfield:V} for $\sum_{E \in \mesh_\ell} \|\nabla u_{\slp,E}^\near\|^2_{L^2(U_E)}$.
\end{proof}

%-------------------------------------------------------------------------
\subsubsection{Estimates for the far field part $u_{\slp,E}^\far$}
\label{sec:farfield-slp}
%-------------------------------------------------------------------------
The following lemma is taken from~{\cite{fkmp}}.
For the convenience of the reader and since the same argument underlies 
the proof of the analogous lemma regarding
the double-layer potential, we recall its proof here.

%-----------------------------------------------------------------------------------------
\def\DD{ {\mathcal D }}
\def\dx{\,dx}
\def\y{ {\mathbf y}}
% === Lemma: caccioppoli Inequality For SLP =====================================
\begin{lemma}[Caccioppoli inequality for $u_{\slp,E}^\far$]
\label{lemma:caccioppoli:V}
There is a constant $\setc{caccioppoli}>0$ depending only on the 
$\kappa$-shape regularity of $\mesh_\ell$ such that for the 
function $u_{\slp,E}^\far$ of 
\eqref{eq:uslp_far_uslp_near} the following is true:
$u_{\slp,E}^\far|_\Omega \in C^\infty(\Omega)$,
$u_{\slp,E}^\far|_{\Omega^\e} \in C^\infty(\Omega^\e)$,
and $u_{\slp,E}^\far|_{U_E} \in C^\infty(U_E)$ with
\begin{align}\label{eq:caccioppoli:V}
  \norm{D^2u_{\slp,E}^\far}{L^2(B_{\delta h_\ell(E)}(x))}
  \le \c{caccioppoli}\,\frac{1}{h_\ell(E)}\,\norm{\nabla u_{\slp,E}^\far}{L^2(B_{2\delta h_\ell(E)}(x))}
  \quad\text{for all }x\in E\in \mesh_\ell.
\end{align}
\end{lemma}
% ==============================================================================
%
% === Proof: caccioppoli Inequality For SLP =====================================
\begin{proof}
  The statements $u_{\slp,\el}^\far|_\Omega \in C^\infty(\Omega)$ and 
  $u_{\slp,\el}^\far|_{\Omega^\e} \in C^\infty(\Omega^\e)$ are taken from \cite[Theorem~{3.1.1}]{ss},
  and we therefore focus on the statements on 
  $u_{\slp,\el}^\far|_{E_E} \in C^\infty(U_E)$ and the estimate 
\eqref{eq:caccioppoli:V}. According to~\cite[Proposition~{3.1.7}]{ss},~\cite[Theorem~{3.1.16}]{ss}, 
and~\cite[Theorem~{3.3.1}]{ss}, the function 
\revision{$u_{\slp,E}^\far\in H^1_{\loc}(\R^d):=\set{v:\R^d\to\R}{v|_K\in H^1(K)\text{ for all }K\subset\R^d
\text{ compact}}$}
solves the transmission problem
  \begin{align}\label{lem:farfield:eq:transmission}
  \begin{array}{rcll}
    -\Delta u_{\slp,E}^\far &=& 0 \quad &\text{a.e. in } \Omega \cup \Omega^\e \\{}
    [u_{\slp,E}^\far] &=& 0 &\text{in } H^{1/2}(\Gamma)\\{}
    [\gamma_1 u_{\slp,E}^\far] &=& -\psi\chi_{\Gamma\setminus U_E}
    &\text{in } H^{-1/2}(\Gamma).
  \end{array}
  \end{align}
In particular, \eqref{lem:farfield:eq:transmission} states that the jump of 
$u_{\slp,E}^\far$ as well as the jump of the normal derivative vanish
on $\Gamma \cap U_E$. This implies that $u_{\slp,E}$ is harmonic in $U_E$
by the following classical argument: First, we observe that $u_{\slp,E}^\far$ 
is distributionally harmonic in $U_E$, since a two-fold  
integration by parts that uses these jump conditions shows 
for 
$v \in C^\infty_0(U_E)$ that 
$\langle u_{\slp,E}^\far,-\Delta v\rangle = 0$.
Weyl's lemma (see, e.g., \cite[Theorem~{2.3.1}]{morrey66}) then implies that 
$u_{\slp,E}^\far$ is therefore strongly harmonic and $
u_{\slp,E}^\far \in C^\infty(U_E)$. 

  The Caccioppoli inequality \eqref{eq:caccioppoli:V} now expresses 
  interior regularity for elliptic problems. Indeed, 
  \cite[Lemma~{5.7.1}]{morrey66} shows 
  \begin{align}
\label{est:inverse_est_cf_Morrey}
   \norm{D^2 u}{L^2(B_r)}
   \lesssim \Big( \norm{f}{L^2(B_{r+h})}
   + \frac{1}{h}\,\norm{\nabla u}{L^2(B_{r+h})}
   + \frac{1}{h^2}\norm{u}{L^2(B_{r+h})}\Big)
  \end{align}
  for each $u\in H^1(B_{r+h})$ such that $u\in H^2(B_r)$ and $\Delta u=f$ on $B_{r+h}$
with balls
$\overline{B}_r\subset B_{r+h}$
with radii $0<r<r+h$ and some $f\in L^2(B_{r+h})$;
\revision{the hidden constant} depends solely on the spatial dimension and is independent of 
  $r$, $h>0$, and $u$, $f$.
  We apply \eqref{est:inverse_est_cf_Morrey} with $f = 0$ and $u = u_{\slp,E}^\far-c_E$,
where $c_E = \frac{1}{|B_{2\delta h_\ell(E)}( x)|}\int_{B_{2\delta h_\ell(E)}( x)}u_{\slp,E}^\far(y)\,dy$.
Using additionally a Poincar\'e inequality then leads to \eqref{eq:caccioppoli:V}. 
\end{proof}
% ==============================================================================
The nonlocal character of the operator $\widetilde \slp$ is represented by the
far field part. 
Lemma~\ref{lemma:caccioppoli:V} allows us to show a local inverse estimate 
for the far field part of the simple-layer operator: 

\begin{lemma}[Local far field bound for $\widetilde\slp$]\label{lemma:farfield:V}
  For all $E\in\mesh_\ell$, there holds 
\begin{align}
\label{eq:farfield:V}
\norm{h_\ell^{1/2}\nabla_\Gamma \gamma_0^\interior u_{\slp,E}^\far}{L^2(E)}
\le \norm{h_\ell^{1/2}\nabla u_{\slp,E}^\far}{L^2(E)}
\le \c{farfield}\,\norm{\nabla u_{\slp,E}^\far}{L^2(U_E)}. 
\end{align}
  The constant $\setc{farfield}>0$ depends only on $\Gamma$ and the $\kappa$-shape regularity constant of $\mesh_\ell$.
\end{lemma}

\begin{proof}
By Lemma~\ref{lemma:caccioppoli:V} we have $u_{\slp,E}^\far \in C^\infty(U_E)$. 
The first estimate in \eqref{eq:farfield:V} follows from the fact
that, for smooth functions, the surface gradient $\nabla_\Gamma(\cdot)$ is the orthogonal projection 
of the gradient $\nabla(\cdot)$ onto the tangent plane, i.e.,
$\nabla_\Gamma \gamma_0^\interior u(x) = \nabla u(x) - \big(\nabla u(x)\cdot\normal(x)\big)\,\normal(x)$, see~\cite{verchota}.

\revision{To prove the second estimate in~\eqref{eq:farfield:V}}, we fix an $E \in \mesh_\ell$. First, we select $N$ points 
$x_j\in E$, $j=1,\ldots,N$,
such that 
$$
E \subset \bigcup_{j=1}^N B_{\delta h_\ell(E)}(x_j). 
$$
We may assume that the \revision{number $N$ of points} depends solely on 
the $\kappa$-shape regularity of the mesh. This follows by geometric
considerations as detailed in \cite[Lemma~{3.5}]{fkmp}; essentially, 
one may select the points $x_j$ from a regular grid with spacing
$\frac{1}{2}\delta h_\ell(E)$ and cover $E$ with balls of radii 
$\delta h_\ell(E)$ centered at these points. The centers outside $E$ 
that are required for the covering are then projected into $E$ 
to ensure that all points are in $E$. 

Next, let $B_i := B_{\delta h_\ell(E)}( x_i)$ and
  $\wat B_{i} := B_{2\delta h_\ell(E)}( x_{i})\subset U_E$.
  Using a standard trace inequality and the Caccioppoli inequality~\eqref{eq:caccioppoli:V}
  we infer for all indices $i$ that
\begin{align*}
\norm{\nabla u_{\slp,\el}^\far}{L^2(B_{i}\cap E)}^2
\lesssim \frac{1}{h_\ell(E)} \norm{\nabla u_{\slp,\el}^\far}{L^2(B_{i})}^2 +
\norm{\nabla u_{\slp,\el}^\far}{L^2(B_i)}\norm{D^2u_{\slp,\el}^\far}{L^2(B_{i})}
\lesssim \frac{1}{h_\ell(E)} \norm{\nabla u_{\slp,\el}^\far}{L^2(\wat B_{i})}^2 .
\end{align*}
We use the last estimate to get
\begin{align*}
\norm{\nabla_\Gamma \gamma_0^\interior u_\el^\far}{L^2(E)}^2
\le \norm{\nabla u_{\slp,\el}^\far}{L^2(E)}^2
\le \sum_{i=1}^N\norm{\nabla u_{\slp,\el}^\far}{L^2(B_{i}\cap E)}^2
&
\lesssim \frac{1}{h_\ell(E)}\,\sum_{i=1}^N\norm{\nabla u_{\slp,\el}^\far}{L^2(\wat B_{i})}^2
\\ &
\lesssim \frac{1}{h_\ell(E)}\, \norm{\nabla u_{\slp,\el}^\far}{L^2(U_E)}^2.
\end{align*}
  This concludes the proof of~\eqref{eq:farfield:V}.
\end{proof}

%-------------------------------------------------------------------------------------------------
Summation of the elementwise estimates of Lemma~\ref{lemma:farfield:V} yields 
the following result:  
\begin{proposition}[Far field bound for $\widetilde\slp$]\label{prop:farfield:V}
  There is a constant $\c{farfield}>0$ depending only on $\Gamma$ and the $\kappa$-shape regularity 
of $\mesh_\ell$ such that 
\begin{align*}
\sum_{E\in\mesh_\ell}\norm{h_\ell^{1/2}\nabla_\Gamma \gamma_0^\interior u_{\slp,E}^\far}{L^2(E)}^2
\le \sum_{E\in\mesh_\ell}\norm{h_\ell^{1/2}\nabla u_{\slp,E}^\far}{L^2(E)}^2
\leq \c{farfield} \left(\norm{\psi}{H^{-1/2}(\Gamma)}^2 + \norm{h_\ell^{1/2}\psi}{L^2(\Gamma)}^2\right).
\end{align*}
\end{proposition}
\begin{proof}
  We use the local far field bound~\eqref{eq:farfield:V}
  of Lemma~\ref{lemma:farfield:V} and $u_{\slp,E}^\far = \widetilde\slp\psi - u_{\slp,E}^\near$,
\begin{align}
\begin{split}
\label{eq2:invest}
&
\sum_{E\in\mesh_\ell}\norm{h_\ell^{1/2}\nabla_\Gamma \gamma_0^\interior u_{\slp,E}^\far}{L^2(E)}^2
\leq \sum_{E\in\mesh_\ell}\norm{h_\ell^{1/2}\nabla u_{\slp,E}^\far}{L^2(E)}^2
\lesssim \sum_{E\in\mesh_\ell} \norm{\nabla u_{\slp,E}^\far}{L^2(U_E)}^2
\\ & \quad
\lesssim \sum_{E\in\mesh_\ell} \norm{\nabla \widetilde\slp\psi}{L^2(U_E)}^2
+\sum_{E\in\mesh_\ell} \norm{\nabla u_{\slp,E}^\near}{L^2(U_E)}^2.
\end{split}
\end{align}
The first term on the right-hand side in~\eqref{eq2:invest} is estimated by stability of $\widetilde\slp$
and the finite overlap property~\eqref{UE:overlap}
  \begin{align*}
    \sum_{E\in\mesh_\ell}\norm{\nabla \widetilde\slp\psi}{L^2(U_E)}^2
    \lesssim \norm{\nabla \widetilde\slp\psi}{L^2(U)}^2
    \lesssim \norm{\psi}{H^{-1/2}(\Gamma)}^2.
  \end{align*}
  The second term in~\eqref{eq2:invest} is bounded with the aid of the near field bound~\eqref{eq1:nearfield:V}.
\end{proof}
%-------------------------------------------------------------------------
\subsection{Far field and near field estimates for the double layer potential}
\label{sec:aux-dlp}
%-------------------------------------------------------------------------
\def\NN{\mathcal N}
%------------------------------------------------------------------------------
Section~\ref{section:invest:aux} studied the simple layer potential in detail. 
Corresponding results for the double layer potential are derived in the present 
section. 
%-------------------------------------------------------------------------
\subsubsection{Decomposition into near field and far field}
%-------------------------------------------------------------------------
We use the notation introduced in Section~\ref{sec:notation-slp}. 
Additionally, in order to define the near field and far field parts for 
the double-layer potential, we need an appropriate cut-off function:
Let $\NN_\ell$ denote the set of nodes of $\mesh_\ell$.
For any $E \in \mesh_\ell$ define
\begin{align}
\label{eq:cut-off-function}
 \eta_E := \sum_{z\in\NN_\ell\cap\omega_\ell(E)}\eta_z,
\end{align}
where $\eta_z\in\SS^1(\mesh_\ell)$ denotes the hat function associated with the boundary node $z$,
i.e., $\eta_z \in \SS^1(\mesh_\ell)$ is characterized by the condition
$\eta_z(z')=\delta_{zz'}$ for all $z'\in\NN_\ell$, 
\revision{%
where $\delta_{zz'}$ denotes the Kronecker delta. 
We use the abbreviation $\widehat\omega_\ell(E):=\omega_\ell(\omega_\ell(E)) := \bigcup\set{E'\in\mesh_\ell}{E'\cap\omega_\ell(E)\neq\emptyset}$
for the second order patch,
where we recall from \eqref{UE:patch} that $\omega_\ell(E)$ denotes the patch of $E\in\mesh_\ell$. 
Note that $\widehat\omega_\ell(E) = \supp(\eta_E)$ for the cut-off function $\eta_E$ of~\eqref{eq:cut-off-function}.
}
We note 
\begin{align}
\label{eta:properties}
 \eta_E|_{\Gamma\cap U_E} = 1,\quad
 \eta_E|_{\Gamma\backslash\omega_\ell(\omega_\ell(E))} = 0,\quad
 \norm{\eta_E}{L^\infty(\Gamma)} = 1,\quad
 \text{and}\quad
 \norm{\nabla_\Gamma\eta_E}{L^\infty(\Gamma)} \simeq h_\ell(E)^{-1},
\end{align}
where the constant involved in the last estimate depends only on the
$\kappa$-shape regularity of $\mesh_\ell$.
For the double-layer potential $u_\dlp = \widetilde\dlp v$ of a density $v\in H^1(\Gamma)$
we define the near field and the far field part by
\begin{align}\label{eq:c_E}
 u_{\dlp,E}^\near := \widetilde\dlp\big((v-c_E)\eta_E\big)
 \quad\text{and}\quad
 u_{\dlp,E}^\far := \widetilde\dlp\big((v-c_E)(1-\eta_E)\big),
\end{align}
where $c_E\in\R$ is a constant that will be specified below.
Since $\widetilde \dlp 1 \equiv -1$ in $\Omega$ and 
$\widetilde \dlp 1 \equiv 0$ in $\Omega^\e$, we have, for every $\el \in \mesh_\ell$,
the identities 
\begin{align}
\label{eq:addition of near and far field contributions of Ktilde}
\begin{split}
 u_\dlp + c_E &= u_{\dlp,E}^\near + u_{\dlp,E}^\far
 \quad\text{in }\Omega
\quad\text{and}\quad
 u_\dlp = u_{\dlp,E}^\near + u_{\dlp,E}^\far
 \quad\text{in }\Omega^\e.
\end{split}
\end{align}
%---------------------------------------------------------------------------------------------
\subsubsection{Inverse estimates for the near field part $u_{\dlp,E}^\near$}
%----------------------------------------------------------------------------------------------------
The proof of the near field bound for the double-layer potential needs an appropriate
choice of the constants $c_E\in\R$ in~\eqref{eq:c_E}.

\begin{lemma}[Poincar\'e inequality on patches]
\label{lemma:poincare}
For given $w\in H^1(\Gamma)$ and $E\in\mesh_\ell$, there is a constant $c_E\in\R$ such that
\begin{align}
\label{est:poincare on patch - first est}
\norm{w-c_E}{L^2(\widehat\omega_\ell(E))}
&
\le \c{poincare} \norm{h_\ell\nabla_\Gamma w}{L^2(\widehat\omega_\ell(E))}, 
\\
\label{est:poincare on patch - second est}
\norm{(w-c_\el)\eta_\el}{H^{1/2}(\Gamma)}
&
\le \c{poincare} \norm{h_\ell^{1/2}\nabla_\Gamma w}{L^2(\widehat\omega_\ell(\el))},
\\
\label{est:poincare on patch - third est}
\norm{(w-c_\el)\eta_\el}{H^{1}(\Gamma)}
&\le \c{poincare} \norm{\nabla_\Gamma w}{L^2(\widehat\omega_\ell(\el))}.
\end{align}
The constant $\setc{poincare}>0$ depends only on the $\kappa$-shape regularity constant of $\mesh_\ell$ and on the surface measure of $\Gamma$.
\end{lemma}

\begin{proof}
The first estimate~\eqref{est:poincare on patch - first est} is established by assembling local Poincar\'e inequalities on $\widehat\omega_\ell(E)$ with the help of~\cite[Theorem~7.1]{ds80}.
The properties of the cut-off function $\eta_E$ detailed in~\eqref{eta:properties}, 
the estimate~\eqref{est:poincare on patch - first est}, and the product rule yield 
\begin{align*}
 \norm{\nabla_\Gamma\big((w-c_E)\eta_E\big)}{L^2(\Gamma)}
 \le \norm{(w-c_E)\nabla_\Gamma\eta_E}{L^2(\Gamma)}
 + \norm{\eta_E\nabla_\Gamma(w-c_E)}{L^2(\Gamma)}
 \lesssim\norm{\nabla_\Gamma w}{L^2(\widehat\omega_\ell(\el))},
\end{align*}
since $\widehat\omega_\ell(\el) = \supp(\eta_\el)$. Hence, we obtain 
with the trivial bound $h_\ell(E) \leq |\Gamma|^{1/(d-1)} \lesssim 1$
\begin{align*}
 \norm{(w-c_\el)\eta_\el}{H^{1}(\Gamma)}^2
 = \norm{(w-c_\el)\eta_\el}{L^2(\Gamma)}^2
 + \norm{\nabla_\Gamma\big((w-c_\el)\eta_\el\big)}{L^2(\Gamma)}^2
 \lesssim \norm{\nabla_\Gamma w}{L^2(\widehat\omega_\ell(\el))}^2.
\end{align*}
This proves~\eqref{est:poincare on patch - third est}.
It remains to verify~\eqref{est:poincare on patch - second est}. To that end, we recall the interpolation inequality
$\norm{\cdot}{H^{1/2}(\Gamma)}^2 \lesssim \norm{\cdot}{L^2(\Gamma)}\norm{\cdot}{H^1(\Gamma)}$
and note that $\norm{(w-c_E)\eta_E}{L^2(\Gamma)} \le \norm{w-c_E}{L^2(\widehat\omega_\ell(E))}$ to get 
\begin{align*}
 \norm{(w-c_\el)\eta_\el}{H^{1/2}(\Gamma)}
 &\lesssim \norm{(w-c_\el)\eta_\el}{L^2(\Gamma)}^{1/2}
 \norm{(w-c_\el)\eta_\el}{H^1(\Gamma)}^{1/2}
 \\&
 \lesssim \norm{h_\ell\nabla_\Gamma w}{L^2(\widehat\omega_\ell(\el))}^{1/2}
 \norm{\nabla_\Gamma w}{L^2(\widehat\omega_\ell(\el))}^{1/2}
 \\&
 \simeq \norm{h_\ell^{1/2}\nabla_\Gamma w}{L^2(\widehat\omega_\ell(\el))},
\end{align*}
where the last estimate hinges on $\kappa$-shape regularity of $\mesh_\ell$.
\end{proof}

%---------------------------------------------------------------------------------------------
The following lemma provides an estimate for the near field part of the double-layer potential.

\begin{proposition}[Near field bound for $\widetilde\dlp$]
\label{lemma:nearfield:K}
Let $v\in H^{1}(\Gamma)$ and consider $u_{\dlp,E}^\near$ defined by \eqref{eq:c_E} 
with the constant $c_E$ given by Lemma~\ref{lemma:poincare}. Then
$\gamma_0^\interior u_{\dlp,E}^\near\in H^1(\Gamma)$, $u_{\dlp,E}^\near|_\Omega\in H^1(\Omega)$,
and $u_{\dlp,E}^\near|_{U\setminus \overline{\Omega}}\in H^1(U \setminus \overline{\Omega})$  with
\begin{align}
\begin{split}
&\sum_{E\in\mesh_\ell}\Big(
\norm{h_\ell^{1/2}\nabla_\Gamma \gamma_0^\interior u_{\dlp,E}^\near}{L^2(E)}^2
+ \norm{\nabla u_{\dlp,E}^\near}{L^2(U_E\cap\Omega)}^2
+ \norm{\nabla u_{\dlp,E}^\near}{L^2(U_E\cap\Omega^\e)}^2
\Big)
\\ & \quad
\le \c{nearfield}\,\norm{h_\ell^{1/2}\nabla_\Gamma v}{L^2(\Gamma)}^2.
\end{split}
\end{align}
The constant $\c{nearfield}>0$ depends only on $\Gamma$
and the $\kappa$-shape regularity of $\mesh_\ell$. 
\end{proposition}

\begin{proof}
  First, the trace of the double-layer potential $\trace\widetilde\dlp = \dlp-\frac{1}{2}: H^1(\Gamma)\rightarrow H^1(\Gamma)$ is continuous, \eqref{def:dlp}. 
  Taking into account~\eqref{eta:properties}
  and the Poincar\'e-type estimate~\eqref{est:poincare on patch - third est}, we observe
\begin{align*}
\norm{\nabla_\Gamma \gamma_0^\interior u_{\dlp,\el}^\near}{L^2(\el)}
&
\leq \norm{\nabla_\Gamma \gamma_0^\interior u_{\dlp,\el}^\near}{L^2(\Gamma)}
\lesssim \norm{(v-c_\el)\eta_\el}{H^1(\widehat\omega_\ell(\el))}
\lesssim \norm{\nabla_\Gamma v}{L^2(\widehat\omega_\ell(E))}.
\end{align*}
  Summation over all $\el\in\mesh_\ell$ shows
\begin{align}\label{eq1:nearfield:K}
\sum_{E\in\mesh_\ell}\norm{h_\ell^{1/2}\nabla_\Gamma \gamma_0^\interior u_{\dlp,E}^\near}{L^2(E)}^2
\lesssim \norm{h_\ell^{1/2}\nabla_\Gamma v}{L^2(\Gamma)}^2.
\end{align}
Second, we use continuity of $\widetilde\dlp: H^{1/2}(\Gamma) \to H^1(U\setminus\Gamma)$ of 
\eqref{eq:mapping-properties-potentials} and get 
\begin{align*}
\norm{\nabla u_{\dlp,\el}^\near}{L^2(U_\el\cap\Omega)}^2 +
\norm{\nabla u_{\dlp,\el}^\near}{L^2(U_\el\cap\Omega^\e)}^2
\lesssim \norm{(v-c_\el)\eta_\el}{H^{1/2}(\Gamma)}^2
\lesssim \norm{h_\ell^{1/2}\nabla_\Gamma v}{L^2(\widehat\omega_\ell(\el))}^2,
  \end{align*}
where we have used~\eqref{est:poincare on patch - second est} in the last step.
  Summation over all $\el\in\mesh_\ell$ gives
\begin{align}\label{eq2:nearfield:K}
\sum_{E\in\mesh_\ell}\Big(\norm{\nabla u_{\dlp,E}^\near}{L^2(U_E\cap\Omega)}^2
  + \norm{\nabla u_{\dlp,E}^\near&}{L^2(U_E\cap\Omega^\e)}^2
  \Big)
  \lesssim \norm{h_\ell^{1/2}\nabla_\Gamma v}{L^2(\Gamma)}^2.
  \end{align}
  Combining~\eqref{eq1:nearfield:K}--\eqref{eq2:nearfield:K}, we conclude the proof.
\end{proof}
%----------------------------------------------------------------------
\subsubsection{Estimates for the far field part $u_{\dlp,\el}^\far$}
%----------------------------------------------------------------------
As for the simple-layer potential, we have a Caccioppoli inequality for the double-layer potential, 
which underlies the analysis of the far field contribution. 

%--------------------------------------------------------------------------------------------------
\begin{lemma}[Caccioppoli inequality for $u_{\dlp,E}^\far$]
\label{lemma:caccioppoli:K}
For the constant $\c{caccioppoli}$ of Lemma~\ref{lemma:caccioppoli:V} 
the functions $u_{\dlp,E}^\far$ of \eqref{eq:c_E} satisfy
  $u_{\dlp,E}^\far|_\Omega \in C^\infty(\Omega)$,
  $u_{\dlp,E}^\far|_{\Omega^\e} \in C^\infty(\Omega^\e)$, and 
  $u_{\dlp,E}^\far|_{U_E}\in C^\infty(U_E)$ with 
  \begin{align}\label{eq:caccioppoli:K}
    \norm{D^2u_{\dlp,E}^\far}{L^2(B_{\delta h_\ell(E)}( x))}
     \le \c{caccioppoli}\,\frac{1}{h_\ell(E)}\,\norm{\nabla u_{\dlp,E}^\far}{L^2(B_{2\delta h_\ell(E)}( x))}
     \quad\text{for all } x\in E\in\mesh_\ell.
  \end{align}
\end{lemma}

\begin{proof}
  The proof is very similar to that of Lemma~\ref{lemma:caccioppoli:V}. One observes 
  that the far field $u_{\dlp,E}^\far$ solves the transmission problem
  \begin{align*}
  \begin{array}{rcll}
    -\Delta u_{\dlp,E}^\far &=& 0 \quad &\text{a.e. in } \Omega\cup\Omega^\e\\{}
    [u_{\dlp,E}^\far] &=& (v-c_E)(1-\eta_E) &\text{in } H^{1/2}(\Gamma)\\{}
    [\gamma_1 u_{\dlp,E}^\far] &=& 0 &\text{in }H^{-1/2}(\Gamma).
  \end{array}
  \end{align*}
  We note that $(1-\eta_E)|_{\Gamma\cap U_E}=0$ by construction of $\eta_E$ 
  in~\eqref{eq:cut-off-function}. Hence, the same reasoning as in the proof of 
  Lemma~\ref{lemma:caccioppoli:V} can be done to reach the conclusion 
  \eqref{eq:caccioppoli:K}. 
\end{proof}
%---------------------------------------------------------------------------------------------------------------------------
\begin{lemma}[Local far field bound for $\widetilde\dlp$]\label{lemma:farfield:K}
  For all $E\in\mesh_\ell$ there holds 
\begin{align}
\label{eq:farfield:K}
\norm{h_\ell^{1/2}\nabla_\Gamma \gamma_0^\interior u_{\dlp,E}^\far}{L^2(E)}
\le \norm{h_\ell^{1/2}\nabla u_{\dlp,E}^\far}{L^2(E)}
\le \c{farfield}\,\norm{\nabla u_{\dlp,E}^\far}{L^2(U_E)}. 
  \end{align}
  The constant $\setc{farfield}>0$ depends only on $\Gamma$ and the $\kappa$-shape regularity constant of $\mesh_\ell$.
\end{lemma}

\begin{proof}
The lemma is shown in exactly the same way as the corresponding bound for 
the simple layer potential $\slp$ in Lemma~\ref{lemma:farfield:V} appealing
to the Caccioppoli inequality \eqref{eq:caccioppoli:K} instead of \eqref{eq:caccioppoli:V}. 
\end{proof}
%---------------------------------------------------------------------------------------------
\begin{proposition}[Far field bound for $\widetilde\dlp$]\label{prop:farfield:K}
  There holds
\begin{align*}
\sum_{\el\in\mesh_\ell}\norm{h_\ell^{1/2}\nabla_\Gamma \gamma_0^\interior u_{\dlp,\el}^\far}{L^2(\el)}^2
\leq \sum_{\el\in\mesh_\ell}\norm{h_\ell^{1/2}\nabla u_{\dlp,\el}^\far}{L^2(\el)}^2
\leq \c{farfield} \left(\norm{h_\ell^{1/2}\nabla_\Gamma v}{L^2(\Gamma)}^2
+ \norm{v}{H^{1/2}(\Gamma)}^2\right).
  \end{align*}
  The constant $\c{farfield}>0$ depends only on $\Gamma$ and the $\kappa$-shape regularity constant
of $\mesh_\ell$.
\end{proposition}

\begin{proof}
We have by Lemma~\ref{lemma:farfield:K} 
\begin{align}
\label{prop:invest:K:eq:2}
\begin{split}
\sum_{\el\in\mesh_\ell}\norm{h_\ell^{1/2}\nabla_\Gamma \gamma_0^\interior u_{\dlp,\el}^\far}{L^2(\el)}^2
&
\leq \sum_{\el\in\mesh_\ell}\norm{h_\ell^{1/2}\nabla u_{\dlp,\el}^\far}{L^2(\el)}^2
\lesssim \sum_{\el\in\mesh_\ell}\norm{\nabla u_{\dlp,\el}^\far}{L^2(U_\el)}^2
\\ &
= \sum_{\el\in\mesh_\ell}\norm{\nabla u_{\dlp,\el}^\far}{L^2(U_\el\cap\Omega)}^2
+ \sum_{\el\in\mesh_\ell}\norm{\nabla u_{\dlp,\el}^\far}{L^2(U_\el\cap\Omega^\e)}^2.
\end{split}
\end{align}
With
the properties in \eqref{eq:addition of near and far field contributions of Ktilde}
 and a triangle inequality we get
\begin{align}
\begin{split}
\sum_{\el\in\mesh_\ell}\norm{h_\ell^{1/2}\nabla u_{\dlp,\el}^\far}{L^2(\el)}^2
\lesssim \sum_{\el\in\mesh_\ell}
&
\Bigl( \norm{\nabla \widetilde\dlp(v-c_\el)}{L^2(U_\el\cap\Omega)}^2
+ \norm{\nabla \widetilde\dlp v}{L^2(U_\el\cap\Omega^\e)}^2\Bigr)
\\ &
+ \sum_{\el\in\mesh_\ell}\Bigl( \norm{\nabla u_{\dlp,\el}^\near}{L^2(U_\el\cap\Omega)}^2
+ \norm{\nabla u_{\dlp,\el}^\near}{L^2(U_\el\cap\Omega^\e)}^2 \Bigr).
\end{split}
\end{align}
The near field contribution is bounded by Proposition~\ref{lemma:nearfield:K}.
Furthermore, noting $\nabla\widetilde\dlp c_\el = \nabla(-c_\el)=0$ in $\Omega$, we get 
\begin{align*}
\sum_{\el\in\mesh_\ell}\norm{h_\ell^{1/2}\nabla u_{\dlp,\el}^\far}{L^2(\el)}^2
&
\lesssim \sum_{\el\in\mesh_\ell}
\Bigl( \norm{\nabla \widetilde\dlp v}{L^2(U_\el\cap\Omega)}^2
+ \norm{\nabla \widetilde\dlp v}{L^2(U_\el\cap\Omega^\e)}^2 \Bigr)
+ \norm{h_\ell^{1/2}\nabla_\Gamma v}{L^2(\Gamma)}^2
\\ &
\lesssim \norm{v}{H^{1/2}(\Gamma)}^2
+ \norm{h_\ell^{1/2}\nabla_\Gamma v}{L^2(\Gamma)}^2,
\end{align*}
where we have used continuity of $\widetilde\dlp$.
\end{proof}
\subsection{Proof of Theorem~\ref{thm:invest} for $\boldsymbol{\gamma=\Gamma}$}
%-------------------------------------------------------------------------
We are now in position to prove the inverse estimates~\eqref{eq:invest:V}--\eqref{eq:invest:W} of Theorem~\ref{thm:invest}.

%========================================================================
\begin{proof}[Proof of the inverse estimate~\eqref{eq:invest:V} for the simple-layer potential $\slp$ and $\gamma=\Gamma$]
Let $\psi \in L^2(\Gamma)$.
Then,
\begin{align}\label{eq1:invest}
\begin{split}
\norm{h^{1/2}\nabla_\Gamma \slp\psi}{L^2(\Gamma)}^2
&
= \sum_{E\in\mesh_\ell} \norm{h_\ell^{1/2}\nabla_\Gamma \slp\psi}{L^2(E)}^2
\\ &
\lesssim \sum_{E\in\mesh_\ell}\norm{h_\ell^{1/2}\nabla_\Gamma \gamma_0^\interior u_{\slp,\el}^\far}{L^2(E)}^2
+ \sum_{E\in\mesh_\ell}\norm{h_\ell^{1/2}\nabla_\Gamma \gamma_0^\interior u_{\slp,\el}^\near}{L^2(E)}^2.
\end{split}
\end{align}
Both sums on the right-hand side can be estimated with the bounds of 
Propositions~\ref{prop:nearfield:V} and~\ref{prop:farfield:V}.
This yields
\begin{align*}
\norm{h_\ell^{1/2}\nabla_\Gamma \slp\psi}{L^2(\Gamma)}
\lesssim \norm{\psi}{H^{-1/2}(\Gamma)}
+ \norm{h_\ell^{1/2}\psi}{L^2(\Gamma)}.
\end{align*}
and concludes the proof.
\end{proof}

\begin{proof}[Proof of the inverse estimate~\eqref{eq:invest:Kadj} for the adjoint double-layer potential $\dlp'$ and $\gamma=\Gamma$]
Let $\psi \in L^2(\Gamma)$.
We split the left-hand side into near field and far field contributions to obtain
\begin{align}
\label{eq:invest:Kadj:1}
\norm{h_\ell^{1/2}\dlp'\psi}{L^2(\Gamma)}^2
\lesssim \sum_{\el\in\mesh_\ell}h_\ell(\el)\norm{\dlp'(\psi\chi_{U_\el\cap\Gamma})}{L^2(\el)}^2 +
\sum_{\el\in\mesh_\ell}h_\ell(\el)\norm{\dlp'(\psi\chi_{\Gamma\setminus U_\el})}{L^2(\el)}^2.
\end{align}
The continuity $\dlp':L^2(\Gamma)\rightarrow L^2(\Gamma)$ stated in \eqref{def:adlp}  yields 
for the near field contribution
\begin{align*}
\sum_{\el\in\mesh_\ell}h_\ell(\el)\norm{\dlp'(\psi\chi_{U_\el\cap\Gamma})}{L^2(\el)}^2
\leq \sum_{\el\in\mesh_\ell}h_\ell(\el)\norm{\dlp'(\psi\chi_{U_\el\cap\Gamma})}{L^2(\Gamma)}^2
&
\lesssim \sum_{\el\in\mesh_\ell}h_\ell(\el)\norm{\psi}{L^2(U_\el\cap\Gamma)}^2
\\ &
\lesssim \norm{h_\ell^{1/2}\psi}{L^2(\Gamma)}^2.
\end{align*}
For the far field contribution, we write 
$u_{\slp,E}^\far  = \widetilde \slp (\psi \chi_{\Gamma \setminus U_E})$
and note that $\dlp' = -1/2 + \gint \widetilde\slp$
and clearly $(\psi\chi_{\Gamma\backslash U_E})|_E=0$. Therefore, 
on $E$ we have $\dlp' (\psi \chi_{\Gamma\setminus U_E}) = \gint u_{\slp,E}^\far$. 
Furthermore, by the smoothness of $u_{\slp,E}^\far$ near $E$ (see 
Lemma~\ref{lemma:caccioppoli:V}) we have 
$\gint u_{\slp,E}^\far = \partial_{\normal} u_{\slp,E}^\far$ on $E$
(cf. Remark~\ref{rem:gamma_1})
and get 
\begin{align*}
\norm{\dlp'(\psi\chi_{\Gamma\setminus U_\el})}{L^2(\el)}
= \norm{\gint u_{\slp,\el}^\far}{L^2(\el)}
= \norm{\partial_\normal u_{\slp,\el}^\far}{L^2(\el)}
\lesssim \norm{\nabla u_{\slp,\el}^\far}{L^2(\el)}.
\end{align*}
The far field contribution in~\eqref{eq:invest:Kadj:1}
can therefore be bounded by Proposition~\ref{prop:farfield:V}
\begin{align*}
\sum_{\el\in\mesh_\ell}h_\ell(\el)\norm{\dlp'(\psi\chi_{\Gamma\setminus U_\el})}{L^2(\el)}^2
\lesssim \sum_{\el\in\mesh_\ell}\norm{h_\ell^{1/2}\nabla u_{\slp,\el}^\far}{L^2(\el)}^2
\lesssim \norm{h_\ell^{1/2}\psi}{L^2(\Gamma)}^2 + \norm{\psi}{H^{-1/2}(\Gamma)}^2.
\end{align*}
Altogether, this gives
\begin{align*}
\norm{h_\ell^{1/2}\dlp'\psi}{L^2(\Gamma)}
\lesssim \norm{h_\ell^{1/2}\psi}{L^2(\Gamma)} + \norm{\psi}{H^{-1/2}(\Gamma)}
\end{align*}
and concludes the proof.
\end{proof}

\begin{proof}[Proof of inverse estimate~\eqref{eq:invest:K} for the double-layer potential $\dlp$ and $\gamma=\Gamma$]
Let $v \in H^1(\Gamma)$.
We recall the stability of $\dlp = \frac{1}{2}+\trace\widetilde\dlp: H^1(\Gamma)\rightarrow H^1(\Gamma)$, from
which we conclude $\trace\widetilde\dlp v\in H^1(\Gamma)$.
Therefore,
\begin{align}
\label{prop:invest:K:eq:0}
\begin{split}
\norm{h_\ell^{1/2}\nabla_\Gamma \dlp v}{L^2(\Gamma)}
&
= \norm{h_\ell^{1/2}\nabla_\Gamma \big(\mbox{$\frac{1}{2}$}+\trace\widetilde\dlp\big) v}{L^2(\Gamma)}
%\\ &
\leq \frac{1}{2}\norm{h_\ell^{1/2}\nabla_\Gamma v}{L^2(\Gamma)}
+ \norm{h_\ell^{1/2}\nabla_\Gamma \gamma_0^\interior u_\dlp}{L^2(\Gamma)}
\end{split}
\end{align}
with $u_\dlp = \widetilde\dlp v$.
There holds $u_\dlp +c_\el = u_{\dlp,\el}^\near + u_{\dlp,\el}^\far$ in $\Omega$,
cf.~\eqref{eq:addition of near and far field contributions of Ktilde}.
For the second term on the right-hand side in~\eqref{prop:invest:K:eq:0}, we obtain with the constants $c_E$ of~Lemma~\ref{lemma:poincare} that
\begin{align}\label{prop:invest:K:eq:1}
\begin{split}
\norm{h_\ell^{1/2}\nabla_\Gamma \gamma_0^\interior u_\dlp }{L^2(\Gamma)}^2
&
= \sum_{\el\in\mesh_\ell}h_\ell(\el)\norm{\nabla_\Gamma \gamma_0^\interior (u_\dlp +c_\el)}{L^2(\el)}^2
\\ &
\lesssim \sum_{\el\in\mesh_\ell}h_\ell(\el)\norm{\nabla_\Gamma \gamma_0^\interior u_{\dlp,\el}^\near}{L^2(\el)}^2
+ \sum_{\el\in\mesh_\ell}h_\ell(\el)\norm{\nabla_\Gamma \gamma_0^\interior u_{\dlp,\el}^\far}{L^2(\el)}^2.
\end{split}
\end{align}
The first sum can be bounded by Proposition~\ref{lemma:nearfield:K},
whereas the second sum can be bounded by Proposition~\ref{prop:farfield:K}.
Altogether, this yields
\begin{align*}
\norm{h_\ell^{1/2}\nabla_\Gamma \dlp v}{L^2(\Gamma)}
\lesssim \norm{v}{H^{1/2}(\Gamma)} + \norm{h_\ell^{1/2}\nabla_\Gamma v}{L^2(\Gamma)}
\end{align*}
and concludes the proof.
\end{proof}

\begin{proof}[Proof of inverse estimate~\eqref{eq:invest:W} for the hypersingular integral operator $\hyp$ and $\gamma=\Gamma$]
Let again $v \in H^1(\Gamma)$.
We split the left-hand side of \eqref{eq:invest:W}
into the sum over all elements $E\in\mesh_\ell$.
On every element, we subtract the constants $c_\el$
from Lemma~\ref{lemma:poincare}.
Note that $\hyp c_E=0$.
Splitting now into near field and far field yields
\begin{align}
\label{eq:invest:W:1}
\norm{h_\ell^{1/2}\hyp v}{L^2(\Gamma)}^2
&
= \sum_{E\in\mesh_\ell}\norm{h_\ell^{1/2}\hyp(v-c_E)}{L^2(E)}^2
\nonumber
\\&
\lesssim \sum_{\el\in\mesh_\ell} \norm{h_\ell^{1/2}\hyp((v-c_\el)\eta_\el)}{L^2(\el)}^2
+ \sum_{\el\in\mesh_\ell}\norm{h_\ell^{1/2}\hyp((v-c_\el)(1-\eta_\el))}{L^2(\el)}^2.
\end{align}
The near field contribution is bounded by the stability of $\hyp:H^1(\Gamma)\rightarrow L^2(\Gamma)$
stated in \eqref{def:hyp}: 
\begin{align*}
\norm{\hyp((v-c_\el)\eta_\el)}{L^2(\el)}^2
\lesssim \norm{( v-c_\el)\eta_\el}{H^1(\widehat\omega_\ell(\el))}^2
\lesssim \norm{\nabla_\Gamma v}{L^2(\widehat\omega_\ell(\el))}^2,
\end{align*}
where we have used the Poincar\'e-type estimate of Lemma~\ref{lemma:poincare}
in the last step.
The sum over all elements gives
\begin{align*}
\sum_{\el\in\mesh_\ell} \norm{h_\ell^{1/2}\hyp((v-c_\el)\eta_\el)}{L^2(\el)}^2
\lesssim \sum_{\el\in\mesh_\ell} h_\ell(\el)\norm{\nabla_\Gamma v}{L^2(\widehat\omega_\ell(\el))}^2
\lesssim \norm{h_\ell^{1/2}\nabla_\Gamma v}{L^2(\Gamma)}^2.
  \end{align*}
It remains to bound the second term on the right-hand side in~\eqref{eq:invest:W:1}.
In view of the support properties of $\eta_E$, the potential 
$u_{\dlp,E}^\far = \widetilde \dlp ((v - c_E)(1-\eta_E))$
is smooth near $E$ (cf. Lemma~\ref{lemma:caccioppoli:K}) so that 
$\gint u_{\dlp,E}^\far = \partial_{\normal} u_{\dlp,E}^\far$ on $E$. 
Furthermore, since 
$\hyp=-\gamma_1^\interior\widetilde\dlp$ we see
\begin{align*}
\norm{\hyp((v-c_\el)(1-\eta_\el))}{L^2(\el)}^2
= \norm{\gint u_{\dlp,\el}^\far}{L^2(\el)}^2
= \norm{\partial_\normal u_{\dlp,\el}^\far}{L^2(\el)}^2
\leq \norm{\nabla u_{\dlp,\el}^\far}{L^2(\el)}^2.
  \end{align*}
We use Proposition~\ref{prop:farfield:K} to conclude
\begin{align*}
\sum_{\el\in\mesh_\ell}\norm{h_\ell^{1/2}\hyp((v-c_\el)(1-\eta_\el))}{L^2(\el)}^2
&
\leq \sum_{\el\in\mesh_\ell}\norm{h_\ell^{1/2}\nabla u_{\dlp,\el}^\far}{L^2(\el)}^2
\lesssim \norm{h_\ell^{1/2}\nabla_\Gamma v}{L^2(\Gamma)}^2
+ \norm{v}{H^{1/2}(\Gamma)}^2.
  \end{align*}
Altogether, we obtain
\begin{align*}
\norm{h_\ell^{1/2}\hyp v}{L^2(\Gamma)}
\lesssim \norm{h_\ell^{1/2}\nabla_\Gamma v}{L^2(\Gamma)}
+ \norm{v}{H^{1/2}(\Gamma)}
  \end{align*}
and thus conclude the proof.
\end{proof}

\subsection{Proof of Theorem~\ref{thm:invest} for $\boldsymbol{\gamma\subsetneqq\Gamma}$}
Finally, it remains to prove the inverse estimates~\eqref{eq:invest:V}--\eqref{eq:invest:W} of Theorem~\ref{thm:invest} for the case $\gamma\subsetneqq\Gamma$.
Let $\psi \in L^2(\gamma)$ and $v \in \H^1(\gamma)$.
We define the trivial extensions
$\widetilde\psi \in L^2(\Gamma)$ and $\widetilde v \in H^1(\Gamma)$ by
\begin{align*}
\widetilde \psi(x):=
\left\{
\begin{array}{cl}
\psi(x) & \text{if }x\in\gamma \\
0 & \text{if }x\in\Gamma\setminus\gamma
\end{array}
\right.,
\qquad
\widetilde v(x):=
\left\{
\begin{array}{cl}
v(x) & \text{if }x\in\gamma \\
0 & \text{if }x\in\Gamma\setminus\gamma
\end{array}
\right..
\end{align*}
Note that $\norm{\psi}{\H^{-1/2}(\gamma)}
= \norm{\widetilde\psi}{H^{-1/2}(\Gamma)}$ and $\norm{v}{\H^{1/2}(\gamma)}
= \norm{\widetilde v}{H^{1/2}(\Gamma)}$.
With this, we see
\begin{align*}
\norm{h_\ell^{1/2}\nabla_\Gamma\slp\psi}{L^2(\gamma)}
\leq \norm{h_\ell^{1/2}\nabla_\Gamma\slp\widetilde\psi}{L^2(\Gamma)}
&
\le\c{inv}\big(\norm{\widetilde\psi}{H^{-1/2}(\Gamma)}
+ \norm{h_\ell^{1/2}\widetilde\psi}{L^2(\Gamma)}\big)
\\ &
=\c{inv}\big(\norm{\psi}{\H^{-1/2}(\gamma)}
+ \norm{h_\ell^{1/2}\psi}{L^2(\gamma)}\big), 
\end{align*}
which proves estimate~\eqref{eq:invest:V}. The other estimates~\eqref{eq:invest:Kadj}--\eqref{eq:invest:W} follow with the same arguments.\qed
% ======================================================================================================
%
% === Proof Of Corollary ===============================================================================

\def\II{\mathcal I}
\def\Cmin{C_{\rm min}}
\def\Cmax{C_{\rm max}}
%====================================================================
\section{Convergent Adaptive Coupling of FEM and BEM}
%====================================================================
\label{section:afb}%
\noindent
In this section, we use the inverse estimates of Section~\ref{section:invest}
to prove convergence of an adaptive FEM-BEM coupling algorithm.

%--------------------------------------------------------------------
\subsection{Model problem}
\label{section:continuous}
%--------------------------------------------------------------------
We consider the following linear interface problem
\begin{subequations}\label{eq:strongform}
\begin{align}
 -\div(\frakA\nabla u^\i) &= f &&\hspace*{-25mm}\text{in }\Omega^\i:=\Omega,\\
 -\Delta u^\e &= 0 &&\hspace*{-25mm}\text{in } \Omega^\e:=\R^d\backslash\overline\Omega,\\
 u^\i-u^\e &= u_0 &&\hspace*{-25mm}\text{on }\Gamma,\\
 (\frakA\nabla u^\i-\nabla u^\e)\cdot\normal &= \phi_0 &&\hspace*{-25mm}\text{on }\Gamma,\\
 u^\e(x) &= \OO(1/|x|) &&\hspace*{-25mm}\text{as }|x|\to\infty,
 \label{eq:strongform:radiation}
\end{align}
\end{subequations}
see Section~\ref{section:solvability} below for some remarks on the radiation condition~\eqref{eq:strongform:radiation} for 2D.
Here, $\Omega$ is a bounded Lipschitz domain in $\R^d$, $d=2,3$, with polygonal resp.\ polyhedral
boundary $\Gamma:=\partial\Omega$ and
exterior
unit normal vector $\normal$.
We assume that the symmetric coefficient matrix $\frakA(x)\in\R^{d\times d}_\sym$
depends Lipschitz continuously on $x$ and has positive and bounded smallest and largest eigenvalues
\begin{align}\label{eq:lambda}
 0<\Cmin\le \lambda_{\rm min}(x)\le \lambda_{\rm max}(x) \le \Cmax<\infty
 \quad\text{for almost all }x\in\Omega
\end{align} and $x$-independent constants $\Cmax \ge \Cmin>0$.
The given data satisfy $f\in L^2(\Omega)$, $u_0\in H^{1/2}(\Gamma)$, and
$\phi_0\in H^{-1/2}(\Gamma)$. As usual,~\eqref{eq:strongform} is understood
in the weak sense, and the sought solutions satisfy
 $u^\i\in H^1(\Omega)$ and $u^\e\in H^1_\loc(\Omega^\e)
=\set{v:\Omega^\e\to\R}{v\in H^1(K)\text{ for all } K\subset\Omega^\e \text{ compact}}$
 with
$\nabla u^\e\in L^2(\Omega^\e)^d$.

With the boundary integral operators from~\eqref{def:slp}--\eqref{def:hyp},
Problem~\eqref{eq:strongform} is equivalently stated via
the symmetric FEM-BEM coupling, cf.\ e.g.~\cite[Theorem~1]{cs1995}:
Find $(u,\phi)\in\HH:=H^1(\Omega)\times H^{-1/2}(\Gamma)$ such that
\begin{subequations}\label{eq:weakform}
\begin{align}
  \dual{\frakA\nabla u}{\nabla v}_\Omega
+ \dual{\hyp u+(\dlp'-\mfrac12)\phi}{v}_\Gamma
  &= \dual{f}{v}_\Omega + \dual{\phi_0+\hyp u_0}{v}_\Gamma,\\
  \dual{\psi}{\slp\phi-(\dlp-\mfrac12)u}_\Gamma
  &= -\dual{\psi}{(\dlp-\mfrac12)u_0}_\Gamma,
\end{align}
\end{subequations}
for all $(v,\psi)\in\HH$.

The link between~\eqref{eq:strongform} and~\eqref{eq:weakform} is provided by
$u=u^\i$ and $\phi = \nabla u^\e\cdot\normal$.
Moreover, $u^\e$ is then given by the third
Green's formula
\begin{align}\label{eq:uext}
 u^\e(x) = \widetilde\dlp(u-u_0)(x) - \widetilde \slp\phi(x) \quad\text{for } x \in \Omega^\e,
\end{align}
where the potentials $\widetilde\slp$ and $\widetilde\dlp$ formally denote the operators $\slp$ and $\dlp$,
but are now evaluated in $\Omega^\e$ instead of $\Gamma$.
Note carefully that we do not use a notational difference
for the function $u\in H^1(\Omega)$ and its trace $u\in H^{1/2}(\Gamma)$, for which
we compute the boundary integrals $\hyp u$ and $(\dlp-\mfrac12)u$ in~\eqref{eq:weakform}.

%--------------------------------------------------------------------
\subsection{Existence and uniqueness of solutions}
\label{section:solvability}
%--------------------------------------------------------------------
Assumption~\eqref{eq:lambda} guarantees pointwise ellipticity
\begin{align}\label{eq:A:elliptic}
 \c{monotone}\,|v-w|^2 \le \big(\frakA(x)v-\frakA(x)w\big)\cdot(v-w)
 \quad\text{for all }v,w\in\R^d\text{ and } x\in\Omega
\end{align}
with $(\cdot)$ denoting the Euclidean scalar product on $\R^d$, as well as
pointwise Lipschitz continuity
\begin{align}\label{eq:A:lipschitz}
 |\frakA(x)v - \frakA(x)w|\le\c{lipschitz}\,|v-w|
 \quad\text{for all }v,w\in\R^d\text{ and }x\in\Omega,
\end{align}
where $\c{monotone}=\Cmin,\c{lipschitz}=\Cmax^{1/2}>0$ do not depend on $x$.

Existence and uniqueness of the solution $\u=(u,\phi)$ of~\eqref{eq:weakform}
rely on the $H^{-1/2}(\Gamma)$-ellipticity of the simple-layer potential
$\slp$. Details are found e.g.\ in~\cite{cs1995,fembem}.
In 3D, the simple-layer potential $\slp$ is always elliptic, i.e.\
\begin{align}\label{eq:V:elliptic}
 \norm{\psi}{H^{-1/2}(\Gamma)}^2 \lesssim \dual{\psi}{\slp\psi}_\Gamma
 \quad\text{for all }\psi\in H^{-1/2}(\Gamma).
\end{align}
To ensure ellipticity in 2D, it suffices to scale $\Omega\subset\R^2$
appropriately so that $\diam(\Omega) < 1$. 

In 2D, the radiation condition $u(x)=\OO(1/|x|)$ either has to be relaxed
to $u(x)=\OO(\log|x|)$ as
$|x|\to\infty$ or the given data must satisfy the compatibility condition
$\int_\Omega f\,dx + \int_\Gamma\phi_0\,d\Gamma = 0$. The latter implies
$\int_\Gamma\phi\,d\Gamma = 0$ and hence the right
decay~\eqref{eq:strongform:radiation} of $u^{\rm ext}$ at infinity, as can
be seen from~\eqref{eq:uext}, cf.\ e.g.~\cite[Section~6.6.1]{s}.

%--------------------------------------------------------------------
\subsection{Galerkin discretization}
%--------------------------------------------------------------------
\label{section:discretization}%
Let $\EE_\ell^\Gamma$ be a $\kappa$-shape regular triangulation of the coupling boundary $\Gamma$
in the sense of Section~\ref{section:invest:preliminaries}. In addition, let
$\TT_\ell$ be a regular triangulation of $\Omega$ into compact and non-degenerate
simplices $T\in\TT_\ell$. As above, we assume that $\Omega$ as well as $\Gamma$
are exactly resolved by $\TT_\ell$ and $\EE_\ell^\Gamma$, and $\kappa$-shape
regularity of $\TT_\ell$ is understood in the sense of
\begin{align}
\sigma(\TT_\ell) := \max_{T\in\TT_\ell}\frac{\diam(T)^d}{|T|}\le\kappa<\infty
\end{align}
with $|\cdot|$ denoting the usual volume measure on $\R^d$.
By $\EE_\ell^\Omega$, we denote the set of facets of $\TT_\ell$ which lie
inside of $\Omega$, but not on the coupling boundary.

For the discretization, we use conforming elements and approximate $u$
by a continuous $\TT_\ell$-piecewise polynomial
$U_\ell\in\SS^p(\TT_\ell)\subset H^1(\Omega)$ of degree $p\ge1$.
Moreover, $\phi$ is approximated by a (possibly discontinuous) $\EE_\ell^\Gamma$-piecewise polynomial
$\Phi_\ell\in\PP^q(\EE_\ell^\Gamma)\subset H^{-1/2}(\Gamma)$ of degree
$q\ge0$.
We stress that the usual link between $p$ and $q$ is $q=p-1$, and the lowest-order case would be $p=1$ and $q=0$.

The discrete spaces read
\begin{align}
 \XX_\ell := \SS^p(\TT_\ell)\times\PP^q(\EE_\ell^\Gamma)
 \quad\subseteq\quad
 H^1(\Omega)\times H^{-1/2}(\Gamma) = \HH,
\end{align}
where the product space $\HH$, equipped with the canonical norm
\begin{align}
 \enorm{\v}
 = \big(\norm{v}{H^1(\Omega)}^2 + \norm{\psi}{H^{-1/2}(\Gamma)}^2\big)^{1/2}
 \quad\text{for }\v:=(v,\psi)\in\HH,
\end{align}
becomes a Hilbert space.

The Galerkin formulation of~\eqref{eq:weakform} reads as follows:
Find $\U_\ell^\star=(U_\ell^\star,\Phi_\ell^\star)\in\XX_\ell$ such that
\begin{subequations}\label{eq:galerkin}
\begin{align}
  \dual{\frakA\nabla U_\ell^\star}{\nabla V_\ell}_\Omega
  + \dual{\hyp U_\ell^\star+(\dlp'-\mfrac12)\Phi_\ell^\star}{V_\ell}_\Gamma
  &= \dual{f}{V_\ell}_\Omega + \dual{\phi_0+\hyp u_0}{V_\ell}_\Gamma,\\
  \dual{\Psi_\ell}{\slp\Phi_\ell^\star-(\dlp-\mfrac12)U_\ell^\star}_\Gamma
  &= -\dual{\Psi_\ell}{(\dlp-\mfrac12)u_0}_\Gamma,
\end{align}
\end{subequations}
for all $\V_\ell=(V_\ell,\Psi_\ell)\in \XX_\ell$.

Again, it is known that~\eqref{eq:galerkin} admits a unique discrete
solution $\U_\ell^\star\in\XX_\ell$ which is quasi-optimal in the sense of the
C\'ea lemma
\begin{align}\label{eq:cea}
 \enorm{\u-\U_\ell^\star}
 \le \c{cea}\,\min_{\V_\ell\in\XX_\ell}\enorm{\u-\V_\ell},
\end{align}
where the constant $\setc{cea}>0$ depends only on $\Omega$,
see~\cite{fembem} resp.~\cite[Appendix]{afp} for the fact that,
contrary to \cite[Corollary~3]{cs1995},
no additional assumption on $\TT_\ell$ or $\EE_\ell^\Gamma$ is needed.

%--------------------------------------------------------------------
\subsection{Perturbed Galerkin discretization}
%--------------------------------------------------------------------
\label{section:perturbed}%
The right-hand side of the discrete formulation~\eqref{eq:galerkin}
involves the evaluation of $\hyp u_0$ and $\dlp u_0$, which can
hardly be performed analytically. Moreover, so-called \emph{fast methods
for boundary integral operators} usually deal with discrete functions,
cf.~\cite{sr2007}.
Therefore, we propose to approximate at least the given boundary data
$u_0\in H^{1/2}(\Gamma)$ by appropriate discrete functions and proceed
analogously to~\cite{kp}: To that end and to provide below a local measure for the approximation error,
we assume additional regularity $u_0\in H^1(\Gamma)$ and use the Scott-Zhang projection
$\JJ_\ell:L^2(\Gamma)\to\SS^p(\EE^\Gamma_\ell)$
from~\cite{scottzhang} to discretize
$U_{0,\ell} = \JJ_\ell u_0\in\SS^p(\EE^\Gamma_\ell)$.
Now, the perturbed Galerkin formulation reads
as follows: Find $\U_\ell=(U_\ell,\Phi_\ell)\in\XX_\ell$ such that
\begin{subequations}\label{eq:perturbed}
\begin{align}
  \dual{\frakA\nabla U_\ell}{\nabla V_\ell}_\Omega
  + \dual{\hyp U_\ell+(\dlp'-\mfrac12)\Phi_\ell}{V_\ell}_\Gamma
  &= \dual{f}{V_\ell}_\Omega + \dual{\phi_0+\hyp U_{0,\ell}}{V_\ell}_\Gamma,\\
  \dual{\Psi_\ell}{\slp\Phi_\ell-(\dlp-\mfrac12)U_\ell}_\Gamma
  &= -\dual{\Psi_\ell}{(\dlp-\mfrac12)U_{0,\ell}}_\Gamma,
\end{align}
\end{subequations}
for all $\V_\ell=(V_\ell,\Psi_\ell)\in \XX_\ell$. Compared
to~\eqref{eq:galerkin}, the only difference is that~\eqref{eq:perturbed}
involves the approximate data $U_{0,\ell}$ instead of $u_0$ on the
right-hand side. Consequently, the same arguments as before prove
that~\eqref{eq:perturbed} has a unique solution. Moreover, the benefit is that~\eqref{eq:perturbed} only involves discrete boundary integral operators, i.e.\ matrices.

\begin{remark}
We stress that the additional regularity assumption $u_0\in H^1(\Gamma)$ is also necessary for the well-posedness of the residual error estimator of Section~\ref{section:aposteriori}.
\end{remark}

%--------------------------------------------------------------------
\subsection{A~posteriori error estimate}
%--------------------------------------------------------------------
\label{section:aposteriori}%
The overall goal of this section is to provide a residual 
{\sl a~posteriori} error estimate~\eqref{eq:indicators} for our 
discretization~\eqref{eq:perturbed}. To that end, we
assume additional regularity $u_0\in H^1(\Gamma)$ and $\phi_0\in L^2(\Gamma)$.
\revision{Recall the notation $\TT_\ell$, $\EE_\ell^\Omega$, and $\EE_\ell^\Gamma$ from Section~\ref{section:discretization}. We}
define the volume residual
\begin{align}\label{eq:eta:volume}
 \eta_\ell(\TT_\ell)^2 = \sum_{T\in\TT_\ell}\eta_\ell(T)^2,
 \quad\text{where}\quad
 \eta_\ell(T)^2 = |T|^{2/d}\, \norm{f+\div(\frakA\nabla U_\ell)}{L^2(T)}^2,
\end{align}
the jump contributions across interior edges
\begin{align}\label{eq:eta:jump}
 \eta_\ell(\EE_\ell^\Omega)^2
 = \sum_{E\in\EE_\ell^\Omega} \eta_\ell(E)^2,
 \quad\text{where}\quad
 \eta_\ell(E)^2 = |E|^{1/(d-1)}\, \norm{[(\frakA\nabla U_\ell)\cdot \normal]}{L^2(E)}^2,
\end{align}
and the boundary contributions on the coupling boundary
\begin{align}\label{eq:eta:boundary}
 \eta_\ell(\EE_\ell^\Gamma)^2
 &= \sum_{E\in\EE_\ell^\Gamma} \eta_\ell(E)^2,
 \quad\text{where}\quad \eta_\ell(E)^2 = \eta_\ell^{(1)}(E)^2 +
 \eta_\ell^{(2)}(E)^2,\nonumber\\
 \begin{split}
 \eta_\ell^{(1)}(E)^2 &= |E|^{1/(d-1)}\, \norm{\phi_0-(\frakA\nabla U_\ell)\cdot\normal + \hyp(U_{0,\ell}-U_\ell)+(\mfrac12-\dlp')\Phi_\ell}{L^2(E)}^2,\\
 \eta_\ell^{(2)}(E)^2 &= |E|^{1/(d-1)}\,\norm{\nabla_\Gamma(\slp\Phi_\ell - (\mfrac12-\dlp)(U_{0,\ell}-U_\ell))}{L^2(E)}^2.
 \end{split}
\end{align}
Moreover, we define the data approximation term
\begin{align}\label{eq:osc}
 \osc_\ell^2 = \sum_{E\in\EE_\ell^\Gamma}\osc_\ell(E)^2,
 \quad\text{where}\quad
 \osc_\ell(E)^2 = |E|^{1/(d-1)}\, \norm{\nabla_\Gamma(u_0-U_{0,\ell})}{L^2(E)}^2.
\end{align}
Altogether, we thus obtain the following {\sl a~posteriori} error bound, whose local contributions are used in Section~\ref{section:algorithm} below to steer an adaptive mesh refinement.

\begin{proposition}\label{prop:aposteriori}
The error is reliably estimated by
\begin{align}\label{eq:indicators}
 \enorm{\u-\U_\ell}^2
 \le\c{reliability} \varrho_\ell^2
 = \!\!\!\!\!\sum_{\tau\in\TT_\ell\cup\EE_\ell^\Omega\cup\EE_\ell^\Gamma}\!\!\!\!\!\varrho_\ell(\tau)^2
 \text{ with }
 \varrho_\ell(\tau)^2
 = \begin{cases}
 \eta_\ell(\tau)^2&\text{for }\tau\in\TT_\ell\cup\EE_\ell^\Omega,\\
 \eta_\ell(\tau)^2 + \osc_\ell(\tau)^2\!\!\!\!&\text{for }\tau\in\EE_\ell^\Gamma.
 \end{cases}
\end{align}
The constant $\setc{reliability}>0$ depends only on $\kappa$-shape regularity of $\TT_\ell$ and $\EE_\ell^\Gamma$.
\end{proposition}

\begin{proof}[Sketch of proof]
Instead of solving the non-perturbed Galerkin formulation~\eqref{eq:galerkin}
of the weak formulation~\eqref{eq:weakform}, we solve the
perturbed Galerkin formulation~\eqref{eq:perturbed} in practice. Put differently,
$\U_\ell=(U_\ell,\Phi_\ell)\in\XX_\ell$ is the Galerkin approximation
of the unique solution $\u_\ell=(u_\ell,\phi_\ell)\in\HH$ of the
continuous perturbed formulation
\begin{subequations}\label{eq:perturbed:weakform}
\begin{align}
  \dual{\frakA\nabla u_\ell}{\nabla v}_\Omega + \dual{\hyp u_\ell+(\dlp'-\mfrac12)\phi_\ell}{v}_\Gamma
  &= \dual{f}{v}_\Omega + \dual{\phi_0+\hyp U_{0,\ell}}{v}_\Gamma,\\
  \dual{\psi}{\slp\phi_\ell-(\dlp-\mfrac12)u_\ell}_\Gamma
  &= -\dual{\psi}{(\dlp-\mfrac12)U_{0,\ell}}_\Gamma,
\end{align}
\end{subequations}
for all $\v=(v,\psi)\in\HH$.
By stability, the error between $\u$ and $\u_\ell$ is controlled by
\begin{align}\label{eq:kp:apx}
 \enorm{\u-\u_\ell}
 \le \norm{u_0-U_{0,\ell}}{H^{1/2}(\Gamma)}
 \lesssim  \norm{h_\ell^{1/2}\nabla_\Gamma(u_0-U_{0,\ell})}{L^2(\Gamma)}
 \simeq \osc_\ell,
\end{align}
where the required approximation estimate for the Scott-Zhang projection can be found 
in~\revision{\cite[Theorem~3]{kp}}. Here, $h_\ell\in\PP^0(\EE_\ell^\Gamma)$ denotes
the local mesh size function of Section~\ref{section:invest:preliminaries}.
Therefore, it only remains to prove
\begin{align*}
 \enorm{\u_\ell-\U_\ell}
 \lesssim \eta_\ell = \Big(\sum_{T\in\TT_\ell}\eta_\ell(T)^2 + \sum_{E\in\EE_\ell^\Omega\cup\EE_\ell^\Gamma}\eta_\ell(E)^2\Big)^{1/2}.
\end{align*}
This follows along the lines of the 2D proof presented in~\revision{\cite[Theorem~4]{cs1995}}. The only
remarkable difference is that in the 3D case the estimate
\begin{align*}
 \norm{\slp\Phi_\ell - (\mfrac12-\dlp)(U_{0,\ell}-U_\ell)}{H^{1/2}(\Gamma)}
 \lesssim
 \norm{h_\ell^{1/2}\nabla_\Gamma\big(\slp\Phi_\ell - (\mfrac12-\dlp)(U_{0,\ell}-U_\ell)\big)}{L^2(\Gamma)}
\end{align*}
does not follow from continuity arguments as in~\cite{cs1995}, but from
a Poincar\'e-type estimate provided by~\revision{\cite[Corollary~4.2]{cms}}.
\end{proof}

\begin{remark}
{\rm(i)}
For the scaling of the different contributions of $\varrho_\ell$ defined in~\eqref{eq:eta:volume}--\eqref{eq:osc},
recall that
$\diam(T)\simeq|T|^{1/d}$ for a $d$-dimensional volume element $T\in\TT_\ell$ and 
$\diam(E)\simeq|E|^{1/(d-1)}$ for a $(d-1)$-dimensional boundary element $E\in\EE_\ell^\Gamma$
or interior facets $E\in\EE_\ell^\Omega$. Altogether, the volume contribution~\eqref{eq:eta:volume} is
weighted by $\diam(T)$, while all other contributions~\eqref{eq:eta:jump}--\eqref{eq:osc} are
weighted by $\diam(E)^{1/2}$.\\
{\rm(ii)} \revision{In~\cite[Theorem~3]{kp}}, the approximation estimate in~\eqref{eq:kp:apx} is proved for any $H^{1/2}$-stable projection
$\JJ_\ell$ onto $\SS^p(\EE_\ell^\Gamma)$. For $p=1$, it is proved in~\revision{\cite[Theorem~6]{kpp}} that newest vertex
bisection guarantees $H^1$-stability 
\revision{and hence also $H^{1/2}$-stability}
of the $L^2$-orthogonal
projection $\Pi_\ell:L^2(\Gamma)\to\SS^1(\EE_\ell^\Gamma)$ on the 2D manifold $\Gamma$. In the case $p=1$, we may thus also
use $U_{0,\ell}=\Pi_\ell u_0$ to discretize the given Dirichlet data, and Proposition~\ref{prop:aposteriori}
still holds accordingly.\\
{\rm(iii)} In 2D, one may also use nodal interpolation $U_{0,\ell}=I_\ell u_0$ to
discretize the data, and~\eqref{eq:kp:apx} holds accordingly. Details are found in~\cite[Proposition~1]{afp}.
\end{remark}

%--------------------------------------------------------------------
\subsection{Local mesh refinement}\label{section:meshrefinement}
%--------------------------------------------------------------------
Let $\TT_\ell$ be a sequence of triangulations of $\Omega$ which is obtained from a given
initial triangulation $\TT_0$ by successive local mesh refinement.
We require the following assumptions on the local mesh refinement of the volume mesh
with $\ell$-independent constants $0<\kappa<\infty$ and $0<q<1$:
\begin{itemize}
\item[(T1)] The triangulations $\TT_\ell$ of $\Omega$ are regular in the sense of Ciarlet and $\kappa$-shape regular.
\item[(T2)] Each refined element $T\in\TT_\ell$ is the disjoint union of its sons $T'\in\TT_{\ell+1}$.
\item[(T3)] The sons $T'\in\TT_{\ell+1}$ of a refined element $T\in\TT_\ell$ satisfy $|T'|\le q\,|T|$.
\item[(T4)] The sons $E'\in\EE_{\ell+1}^\Omega$ of some refined facet $E\in\EE_\ell^\Omega$ satisfy $|E'|\le q\,|E|$.
\end{itemize}
In addition, let $\EE_\ell^\Gamma$ be a sequence of triangulations of the coupling boundary $\Gamma$ which is obtained from an initial triangulation $\EE_0^\Gamma$  by successive local mesh refinement.
For the local mesh refinement of the boundary mesh, we assume the following:
\begin{itemize}
\item[(E1)] The triangulations $\EE_\ell^\Gamma$ of $\Gamma$ are regular in the sense of Ciarlet and $\kappa$-shape regular.
\item[(E2)] Each refined element $E\in\EE_\ell^\Gamma$ is the disjoint union of its sons $E'\in\EE_{\ell+1}$.
\item[(E3)] The sons $E'\in\EE_{\ell+1}^\Gamma$ of a refined element $E\in\EE_\ell^\Gamma$ satisfy $|E'|\le q\,|E|$.
\end{itemize}
From (T2) and~(E2), it follows that the discrete spaces are nested $\XX_{\ell}\subseteq\XX_{\ell+1}$.

All of these assumptions are satisfied for the usual mesh refinement strategies,
see e.g.~\cite{verfuerth96}.
For instance, one may use newest vertex bisection for both $\TT_\ell$ and $\EE_\ell^\Gamma$. We
refer to e.g.~\cite{stevenson} for details on the latter algorithm, but remark that (T3)--(T4) and (E3) are then satisfied with
$q=1/2$. Moreover, we stress that 2D and 3D newest vertex bisection only leads to finitely many equivalence classes
of elements $T\in\bigcup_{\ell=0}^\infty\TT_\ell$.

In the experiments below, we let $\EE_\ell^\Gamma = \TT_\ell|_\Gamma$ be the induced triangulation of
the coupling boundary $\Gamma$ and use newest vertex bisection to refine $\TT_\ell$.
In 3D, $\EE_\ell^\Gamma$ can then equivalently be obtained by use of 2D newest
vertex bisection from $\EE_0^\Gamma$.
We stress, however, that this coupling
is \emph{not} needed in theory.

%--------------------------------------------------------------------
\subsection{Adaptive algorithm and convergence}
%--------------------------------------------------------------------
\label{section:algorithm}%
In this section, we consider the following common adaptive algorithm, which steers the local mesh refinement by use of the local contributions of $\varrho_\ell^2 = \eta_\ell^2 + \osc_\ell^2$ of Proposition~\ref{prop:aposteriori}.

\begin{algorithm}\label{algorithm}
\textsc{Input}: Initial meshes $(\TT_0,\EE_0)$ for $\ell:=0$, adaptivity parameter $\theta\in(0,1)$.
\begin{itemize}
\item[(i)] Compute discrete solution $\U_\ell\in\XX_\ell$.
\item[(ii)] Compute refinement indicators $\varrho_\ell(\tau)$ from~\eqref{eq:indicators} for all $\tau\in\TT_\ell\cup\EE_\ell^\Omega\cup\EE_\ell^\Gamma$.
\item[(iii)] Determine a set $\MM_\ell\subseteq\TT_\ell\cup\EE_\ell^\Omega\cup\EE_\ell^\Gamma$ which satisfies the D\"orfler marking criterion
\begin{align}\label{eq:doerfler}
 \theta\,\varrho_\ell^2
 \le \sum_{T\in\TT_\ell\cap\MM_\ell}\varrho_\ell(T)^2
 + \sum_{E\in\EE_\ell^\Omega\cap\MM_\ell}\varrho_\ell(E)^2
 + \sum_{E\in\EE_\ell^\Gamma\cap\MM_\ell}\varrho_\ell(E)^2.
\end{align}
\item[(iv)] Mark elements $T\in\TT_\ell\cap\MM_\ell$ and facets
$E\in(\EE_\ell^\Omega\cup\EE_\ell^\Gamma)\cap\MM_\ell$ for refinement.
\item[(v)] Generate $(\TT_{\ell+1},\EE_{\ell+1})$ by refinement of at least all marked elements and facets.
\item[(vi)] Increase counter $\ell\mapsto\ell+1$, and goto {\rm(i)}.
\end{itemize}
\textsc{Output}: Sequence of error estimators $(\varrho_\ell)_{\ell\in\N}$ and
discrete solutions $(\U_\ell)_{\ell\in\N}$.\qed
\end{algorithm}

\revision{Since adaptive mesh refinement does not guarantee that the volume mesh $\TT_\ell$ and the boundary mesh $\EE_\ell^\Gamma$ become infinitely fine, convergence $\U_\ell\to\u$ is a~priori unclear as $\ell\to0$. 
However, we employ the concept of \emph{estimator reduction} from~\cite{estconv} to prove $\varrho_\ell\to0$ as $\ell\to0$. Finally, convergence $\U_\ell\to\u$ as $\ell\to\infty$ then follows from the reliability~\eqref{eq:indicators} of $\varrho_\ell$. These observations are stated in the following theorem.}

\begin{theorem}\label{thm:convergence}
Algorithm~\ref{algorithm} yields a sequence of meshes $\TT_\ell$, $\EE_\ell$, and approximations $\U_{\ell}$ 
with the following properties: 
\begin{enumerate}[\rm(i)]
\item 
there are constants $0<\rho<1$ and $\c{estconv}>0$ such that the overall error estimator 
$\varrho_\ell^2 = \eta_\ell^2+\osc_\ell^2$ satisfies
\begin{align}\label{eq:estconv}
 \varrho_{\ell+1}^2
 \le \rho\,\varrho_\ell^2
 + \c{estconv}\,\big(\enorm{\U_{\ell+1}-\U_\ell}^2
 + \norm{U_{0,\ell+1}-U_{0,\ell}}{H^{1/2}(\Gamma)}^2\big)
 \quad\text{for all }\ell\ge0.
\end{align}
The constant $0<\rho<1$ depends only on the mesh size reduction constant $0<q<1$ from assumptions~(T3)--(T4) and (E3) and the adaptivity parameter $0<\theta<1$. The constant $\setc{estconv}>0$ depends additionally on $\Omega$, $\kappa$-shape regularity of $\TT_\ell$ and $\EE_\ell^\Gamma$, and the polynomial degrees $p$, $q$. 
\item 
The error indicators as well as the oscillation terms tend to zero: 
\begin{align}\label{eq:convergence}
 \lim_{\ell\to\infty}\eta_\ell
 = 0 = \lim_{\ell\to\infty}\osc_\ell.
\end{align}
\item 
The adaptive coupling algorithm converges: 
$\lim\limits_{\ell\to\infty}\enorm{\u-\U_\ell} = 0$. 
\end{enumerate}
\end{theorem}

\begin{proof}[Proof of Estimate~\eqref{eq:estconv} of Theorem~\ref{thm:convergence}]
We recall that $(a+b)^2\le (1+\delta)a^2+(1+\delta^{-1})b^2$ for all $a,b\in\R$ and arbitrary $\delta>0$.

$\bullet$ First, for the volume contributions~\eqref{eq:eta:volume} and the interior jumps~\eqref{eq:eta:jump},
we argue as in~\cite{ckns}.
By use of the triangle inequality, we see
\begin{align}\label{eq0:rho:volume}
 \begin{split}
 \eta_{\ell+1}(\TT_{\ell+1})^2
 &= \sum_{T'\in\TT_{\ell+1}}|T'|^{2/d}\,\norm{f+\div(\frakA\nabla U_{\ell+1})}{L^2(T')}^2
 \\&
 \le (1+\delta)\sum_{T'\in\TT_{\ell+1}}|T'|^{2/d}\,\norm{f+\div(\frakA\nabla U_{\ell})}{L^2(T')}^2
 \\&\quad
 + (1+\delta^{-1})\sum_{T'\in\TT_{\ell+1}}|T'|^{2/d}\,\norm{\div \frakA(\nabla U_{\ell+1}-\nabla U_{\ell})}{L^2(T')}^2
 \end{split}
\end{align}
The product rule gives $\div \frakA(\nabla U_{\ell+1}-\nabla U_{\ell}) = (\div\frakA)\cdot\nabla U_{\ell+1} + \frakA:D^2(\nabla U_{\ell+1}-\nabla U_{\ell})$ with $D^2(\cdot)$ being the Hessian matrix. Using the triangle inequality and a scaling argument,
the arguments of the second sum in~\eqref{eq0:rho:volume} are estimated by
\begin{align*}
 |T'|^{2/d}\,&\norm{\div \frakA(\nabla U_{\ell+1}-\nabla U_{\ell})}{L^2(T')}^2
 \\&
 \lesssim |T'|^{2/d}\big(\norm{\div \frakA}{L^\infty(T')}^2 \norm{\nabla(U_{\ell+1}-U_\ell)}{L^2(T')}^2
 + \norm{\frakA}{L^\infty(T')}^2 \norm{D^2(\nabla U_{\ell+1}-\nabla U_{\ell})}{L^2(T')}^2
 \\&
 \lesssim
 \norm{\frakA}{W^{1,\infty}(T')}^2\norm{\nabla(U_{\ell+1}-U_\ell)}{L^2(T')}^2.
\end{align*}
The first sum in~\eqref{eq0:rho:volume} is split into non-refined and refined elements, and assumption~(T3)
is used for the refined elements. Recall that assumption~(T2) states that each element $T\in\TT_\ell$ is the union of its sons $T'\in\TT_{\ell+1}$.
This gives
\begin{align*}
 \sum_{T'\in\TT_{\ell+1}}&|T'|^{2/d}\,\norm{f+\div(\frakA\nabla U_{\ell})}{L^2(T')}^2
 \\&
 \le \sum_{T\in\TT_\ell\cap\TT_{\ell+1}}|T|^{2/d}\,\norm{f+\div(\frakA\nabla U_{\ell})}{L^2(T)}^2
 + q^{2/d}\sum_{T\in\TT_\ell\backslash\TT_{\ell+1}}|T|^{2/d}\,\norm{f+\div(\frakA\nabla U_{\ell})}{L^2(T)}^2
 \\&
 = \eta_\ell(\TT_\ell)^2 - (1-q^{2/d})\,\eta_\ell(\TT_\ell\backslash\TT_{\ell+1})^2.
\end{align*}
Altogether, this gives
\begin{align}\label{eq:rho:volume}
 \begin{split}
 \eta_{\ell+1}(\TT_{\ell+1})^2
 &\le (1+\delta)\big(\eta_\ell(\TT_\ell)^2-(1-q^{2/d})\,\eta_\ell(\TT_\ell\backslash\TT_{\ell+1})^2\big)
 \\&\qquad
 + (1+\delta^{-1})C\,\norm{\nabla(U_{\ell+1}-U_\ell)}{L^2(\Omega)}^2,
 \end{split}
\end{align}
where $C>0$ depends on $\Omega$, on $\kappa$-shape regularity of $\TT_{\ell+1}$, and the (local) $W^{1,\infty}$-norm
of $\frakA$.

$\bullet$ Second, for the interior jumps~\eqref{eq:eta:jump}, we argue 
\begin{align}\label{eq0:rho:jump}
 \begin{split}
 \eta_{\ell+1}(\EE_{\ell+1}^\Omega)^2
 &= \sum_{E'\in\EE_{\ell+1}^\Omega} |E'|^{1/(d-1)}\norm{[\frakA\nabla U_{\ell+1}\cdot\normal]}{L^2(E')}^2
 \\&
 \le (1+\delta) \sum_{E'\in\EE_{\ell+1}^\Omega} |E'|^{1/(d-1)}\norm{[\frakA\nabla U_\ell\cdot\normal]}{L^2(E')}^2
 \\&\quad
 + (1+\delta^{-1}) \sum_{E'\in\EE_{\ell+1}^\Omega} |E'|^{1/(d-1)}\norm{[\frakA\nabla(U_{\ell+1}-U_\ell)\cdot\normal]}{L^2(E')}^2
 \end{split}
\end{align}
The arguments of the second sum in~\eqref{eq0:rho:jump} are estimated by use of a scaling argument, namely,
\begin{align*}
 |E'|^{1/(d-1)}\norm{[\frakA\nabla(U_{\ell+1}-U_\ell)\cdot\normal]}{L^2(E')}^2
 \lesssim \norm{\frakA}{W^{1,\infty}(\omega_{\ell+1,E'})}^2\,\norm{\nabla(U_{\ell+1}-U_\ell)}{L^2(\omega_{\ell+1,E'})}^2
\end{align*}
with $\omega_{\ell+1,E'}=T'_+\cup T'_-$ being the patch of $E'=T'_+\cap T'_-\in\EE_{\ell+1}^\Omega$. To treat the first sum in~\eqref{eq0:rho:jump},
the essential observation
is that $\frakA\nabla U_\ell$ is continuous across new facets $\EE_{\ell+1}^\Omega\backslash\EE_\ell^\Omega$
so that the respective jumps vanish. This and assumption~(T4) for refined interior facets give
\begin{align*}
  &\sum_{E'\in\EE_{\ell+1}^\Omega}|E'|^{1/(d-1)}\norm{[\frakA\nabla U_{\ell}\cdot\normal]}{L^2(E')}^2
 \\&\quad
 \le\sum_{E'\in\EE_\ell^\Omega\cap\EE_{\ell+1}^\Omega} |E|^{1/(d-1)}\norm{[\frakA\nabla U_{\ell}\cdot\normal]}{L^2(E)}^2
 + q^{1/(d-1)}\sum_{E'\in\EE_\ell^\Omega\backslash\EE_{\ell+1}^\Omega}|E|^{1/(d-1)}\norm{[\frakA\nabla U_{\ell}\cdot\normal]}{L^2(E)}^2
 \\&\quad
 = \eta_\ell(\EE_\ell^\Omega)^2 - (1-q^ {1/(d-1)})\eta_\ell(\EE_\ell^\Omega\backslash\EE_{\ell+1}^\Omega)^2.
\end{align*}
Altogether, we
obtain
\begin{align}\label{eq:rho:jump}
 \begin{split}
 \eta_{\ell+1}(\EE_{\ell+1}^\Omega)^2
 &\le (1+\delta)\big(\eta_\ell(\EE_\ell^\Omega)^2 - (1-q^ {1/(d-1)})\eta_\ell(\EE_\ell^\Omega\backslash\EE_{\ell+1}^\Omega)^2\big)
 \\&\qquad
 + (1+\delta^{-1})C\,\norm{\nabla(U_{\ell+1}-U_\ell)}{L^2(\Omega)}^2,
 \end{split}
\end{align}
where $C>0$ depends on $\kappa$-shape regularity of $\TT_{\ell+1}$ and the (local) $W^{1,\infty}$-norm
of $\frakA$.

$\bullet$ Third, we consider the first boundary contribution~\eqref{eq:eta:boundary} of the discretization error estimator.
By use of the triangle inequality, we see
\begin{align*}
 &\norm{\phi_0-(\frakA\nabla U_{\ell+1})\cdot\normal + \hyp(U_{0,\ell+1}-U_{\ell+1})+(\mfrac12-\dlp')\Phi_{\ell+1}}{L^2(E')}
 \\&\quad
 \le \norm{\phi_0-(\frakA\nabla U_\ell)\cdot\normal + \hyp(U_{0,\ell}-U_\ell)+(\mfrac12-\dlp')\Phi_\ell}{L^2(E')}
 \\&\qquad
 + \norm{\frakA\nabla(U_{\ell+1}-U_\ell)\cdot\normal}{L^2(E')}
 + \norm{\hyp\big((U_{0,\ell+1}-U_{0,\ell})-(U_{\ell+1}-U_\ell)\big)}{L^2(E')}
 \\&\qquad
 + \frac12\,\norm{\Phi_{\ell+1}-\Phi_\ell}{L^2(E')}
 + \norm{\dlp'(\Phi_{\ell+1}-\Phi_\ell)}{L^2(E')}
\end{align*}
We sum these terms over all elements $E'\in\EE_{\ell+1}^\Gamma$ and use assumption~(E3) to see
\begin{align*}
 &\sum_{E'\in\EE_{\ell+1}^\Gamma}|E'|^{1/(d-1)}
 \norm{\phi_0-(\frakA\nabla U_\ell)\cdot\normal + \hyp(U_{0,\ell}-U_\ell)+(\mfrac12-\dlp')\Phi_\ell}{L^2(E')}^2
 \\&\quad
 \le \sum_{E\in\EE_{\ell}^\Gamma\cap\EE_{\ell+1}^\Gamma}|E|^{1/(d-1)}
 \norm{\phi_0-(\frakA\nabla U_\ell)\cdot\normal + \hyp(U_{0,\ell}-U_\ell)+(\mfrac12-\dlp')\Phi_\ell}{L^2(E)}^2
 \\&\qquad
 + q^{1/(d-1)}\sum_{E\in\EE_{\ell}^\Gamma\backslash\EE_{\ell+1}^\Gamma}|E|^{1/(d-1)}
  \norm{\phi_0-(\frakA\nabla U_\ell)\cdot\normal + \hyp(U_{0,\ell}-U_\ell)+(\mfrac12-\dlp')\Phi_\ell}{L^2(E)}^2
.
\end{align*}
A scaling argument reveals
\begin{align*}
 \sum_{E'\in\EE_{\ell+1}^\Gamma}|E'|^{1/(d-1)}\norm{\frakA\nabla(U_{\ell+1}-U_\ell)\cdot\normal}{L^2(E')}^2
 \lesssim \norm{\nabla(U_{\ell+1}-U_\ell)}{L^2(\Omega)}^2,
\end{align*}
where the implied constant depends on the (local) $W^{1,\infty}$-norm of $\frakA$ and $\kappa$-shape regularity of $\TT_{\ell+1}$.
The inverse estimates of Corollary~\ref{cor:invest} for $\dlp'$ and $\hyp$ shows
\begin{align*}
 &\sum_{E'\in\EE_{\ell+1}^\Gamma}|E'|^{1/(d-1)}\big(
 \norm{\hyp\big((U_{0,\ell+1}-U_{0,\ell})-(U_{\ell+1}-U_\ell)\big)}{L^2(E')}^2
 + \norm{\dlp'(\Phi_{\ell+1}-\Phi_\ell)}{L^2(E')}^2
\big)
 \\&\quad
 = \norm{h_{\ell+1}^{1/2}\hyp\big((U_{0,\ell+1}-U_{0,\ell})-(U_{\ell+1}-U_\ell)\big)}{L^2(\Gamma)}^2
 + \norm{h_{\ell+1}^{1/2}\dlp'(\Phi_{\ell+1}-\Phi_\ell)}{L^2(\Gamma)}^2
 \\&\quad
 \lesssim \norm{(U_{0,\ell+1}-U_{0,\ell})-(U_{\ell+1}-U_\ell)}{H^{1/2}(\Gamma)}^2
 + \norm{\Phi_{\ell+1}-\Phi_\ell}{H^{-1/2}(\Gamma)}^2.
\end{align*}
An inverse estimate from~\cite[Theorem~3.6]{ghs} proves
\begin{align*}
 \sum_{E'\in\EE_{\ell+1}^\Gamma}|E'|^{1/(d-1)}\norm{\Phi_{\ell+1}-\Phi_\ell}{L^2(E')}
 = \norm{h_{\ell+1}^{1/2}(\Phi_{\ell+1}-\Phi_\ell)}{L^2(\Gamma)}
 \lesssim \norm{\Phi_{\ell+1}-\Phi_\ell}{H^{-1/2}(\Gamma)}^2,
\end{align*}
where the implied constant depends only on $\Gamma$, the $\kappa$-shape regularity
of $\EE_{\ell+1}^\Gamma$, and the polynomial degree $q$. Combining all these estimates, we obtain
\begin{align}\label{eq:rho:boundary1}
\begin{split}
 \eta_{\ell+1}^{(1)}(\EE_{\ell+1}^\Gamma)^2
 &\le (1+\delta)\big(\eta_\ell^{(1)}(\EE_\ell^\Gamma)^2 -(1-q^{1/(d-1)})\eta_\ell^{(1)}(\EE_\ell^\Gamma\backslash\EE_{\ell+1}^\Gamma)^2\big)
 \\&\qquad
 + (1+\delta^{-1})C\,\big(
 \norm{\nabla(U_{\ell+1}-U_\ell)}{L^2(\Omega)}^2
 + \norm{U_{\ell+1}-U_\ell}{H^{1/2}(\Gamma)}^2
 \\&\qquad\qquad\qquad\qquad\qquad
 + \norm{U_{0,\ell+1}-U_{0,\ell}}{H^{1/2}(\Gamma)}^2
 + \norm{\Phi_{\ell+1}-\Phi_\ell}{H^{-1/2}(\Gamma)}^2
 \big)
\end{split}
\end{align}
for the first contribution to $\eta_\ell(\EE_\ell^\Gamma)$ from~\eqref{eq:eta:boundary}.
The constant $C>0$ depends only on $\Gamma$, on $\kappa$-shape regularity of $\EE_{\ell+1}^\Gamma$
and $\TT_{\ell+1}$, and on the polynomial degree $q$.

$\bullet$ Fourth, we consider the second boundary contribution~\eqref{eq:eta:boundary} of the
discretization error estimator. 
Arguing as before, we now use the inverse estimates from 
Corollary~\ref{cor:invest} for $\slp$ and $\dlp$ as well as a local inverse
estimate $\norm{h_\ell\nabla_\Gamma V_\ell}{L^2(\Gamma)} \lesssim 
\norm{V_\ell}{H^{1/2}(\Gamma)}$ for all $V_\ell \in \SS^p(\TT_\ell)$ which
is e.g.\ found in~\cite[Proposition~4.1]{ccdpr:hypsing} for 2D and
\cite[Proposition~3]{akp} for 3D. We obtain
\begin{align}\label{eq:rho:boundary2}
\begin{split}
 \eta_{\ell+1}^{(2)}(\EE_{\ell+1}^\Gamma)^2
 &\le (1+\delta)\big(\eta_\ell^{(2)}(\EE_\ell^\Gamma)^2 -(1-q^{1/(d-1)})\eta_\ell^{(2)}(\EE_\ell^\Gamma\backslash\EE_{\ell+1}^\Gamma)^2\big)
 \\&\qquad
 + (1+\delta^{-1})C\,\big(
 \norm{U_{\ell+1}-U_\ell}{H^{1/2}(\Gamma)}^2
 \\&\qquad\qquad\qquad\qquad\qquad
 + \norm{U_{0,\ell+1}-U_{0,\ell}}{H^{1/2}(\Gamma)}^2
 + \norm{\Phi_{\ell+1}-\Phi_\ell}{H^{-1/2}(\Gamma)}^2
 \big)
\end{split}
\end{align}
for the first contribution to $\eta_\ell(\EE_\ell^\Gamma)$ from~\eqref{eq:eta:boundary}.
The constant $C>0$ depends only on $\kappa$-shape regularity of $\EE_{\ell+1}^\Gamma$
and $\TT_{\ell+1}$, on $\Gamma$, and on the polynomial degree $p$.

$\bullet$ Fifth, we consider the oscillation terms from~\eqref{eq:osc}. Arguing as
in the previous steps, it follows that
\begin{align}\label{eq:rho:osc}
 \begin{split}
 \osc_{\ell+1}^2
 &\le (1+\delta)\big(\osc_\ell^2 - (1-q^{1/(d-1)})\osc_\ell(\EE_\ell^\Gamma\backslash\EE_{\ell+1}^\Gamma)^2\big)
 \\&\qquad
 + (1+\delta^{-1})C\,\norm{U_{0,\ell+1}-U_{0,\ell}}{H^{1/2}(\Gamma)}^2,
 \end{split}
\end{align}
where the constant $C>0$ depends on $\Gamma$, $\kappa$-shape regularity of $\EE_\ell^\Gamma$, and the polynomial degree $p$.

$\bullet$ Sixth, we combine the reduction estimates~\eqref{eq:rho:volume}, \eqref{eq:rho:jump}, \eqref{eq:rho:boundary1}, \eqref{eq:rho:boundary2}, \eqref{eq:rho:osc} for the different parts of $\varrho_\ell$.
By a norm equivalence (with constants depending only on $\Omega$) we have 
$\norm{\nabla(U_{\ell+1}-U_\ell)}{L^2(\Omega)}^2
+\norm{U_{\ell+1}-U_\ell}{H^{1/2}(\Gamma)}^2
\simeq\norm{U_{\ell+1}-U_\ell}{H^1(\Omega)}^2$. Together with $0<q<1$ and hence $q^{2/d}\le q^{1/(d-1)}$, we obtain
\begin{align*}
 \varrho_{\ell+1}^2
 &= \eta_{\ell+1}(\TT_{\ell+1})^2
 + \eta_{\ell+1}(\EE_{\ell+1}^\Omega)^2
 + \eta_{\ell+1}(\EE_{\ell+1}^\Gamma)^2
 + \osc_{\ell+1}^2
 \\&
 \le (1+\delta)\Big(\eta_\ell(\TT_\ell)^2
 + \eta_\ell(\EE_\ell^\Omega)^2
 + \eta_\ell(\EE_\ell^\Gamma)^2
 + \osc_\ell^2
 \\&\qquad
 -(1-q^{1/(d-1)})\big(\eta_\ell(\TT_\ell\backslash\TT_{\ell+1})^2
 +\eta_\ell(\EE_\ell^\Omega\backslash\EE_{\ell+1}^\Omega)^2
 +\eta_\ell(\EE_\ell^\Gamma\backslash\EE_{\ell+1}^\Gamma)^2
 +\osc_\ell(\EE_\ell^\Gamma\backslash\EE_{\ell+1}^\Gamma)^2\big)
 \Big)
 \\&\quad
 +(1+\delta^{-1})C\,\Big(
 \norm{U_{\ell+1}-U_\ell}{H^1(\Omega)}^2
 + \norm{U_{0,\ell+1}-U_{0,\ell}}{H^{1/2}(\Gamma)}^2
 + \norm{\Phi_{\ell+1}-\Phi_\ell}{H^{-1/2}(\Gamma)}^2
 \Big),
\end{align*}
and the constant $C>0$ depends only on $\Omega$, $\kappa$-shape regularity of $\TT_\ell$ and $\EE_\ell^\Gamma$, and the polynomial degrees $p,q$. To abbreviate notation, we define
the index set $\II_\ell:=\TT_\ell\cup\EE_\ell^\Omega\cup\EE_\ell^\Gamma$. Together with norm equivalence $\norm{U_{\ell+1}-U_\ell}{H^1(\Omega)}^2 + \norm{\Phi_{\ell+1}-\Phi_\ell}{H^{-1/2}(\Gamma)}^2
\simeq\enorm{\U_{\ell+1}-\U_\ell}^2$, the previous estimate takes the form
\begin{align}\label{eq:estimator}
 \begin{split}
 \varrho_{\ell+1}^2
 &\le (1+\delta)\big(\varrho_\ell^2
 -(1-q^{1/(d-1)})\,\varrho_\ell(\II_\ell\backslash\II_{\ell+1})^2\big)
 \\&\quad
 +(1+\delta^{-1})C\,\big(\enorm{\U_{\ell+1}-\U_\ell}^2+\norm{U_{0,\ell+1}-U_{0,\ell}}{H^{1/2}(\Gamma)}^2\big).
 \end{split}
\end{align}

$\bullet$ Seventh, recall that marked elements and facets are refined so that $\MM_\ell \subseteq \II_\ell\backslash\II_{\ell+1}$. With the abbreviate notation, the D\"orfler marking~\eqref{eq:doerfler} takes the form
\begin{align*}
 \theta\,\varrho_\ell^2
 \le \varrho_\ell(\MM_\ell)^2
 \le \varrho_\ell(\II_\ell\backslash\II_{\ell+1})^2.
\end{align*}
Therefore, the estimate~\eqref{eq:estimator} becomes
\begin{align*}
 \varrho_{\ell+1}^2
 &= (1+\delta)\big(1 - (1-q^{1/(d-1)})\theta\big)\,\varrho_\ell^2
 +(1+\delta^{-1})C\,\big(\enorm{\U_{\ell+1}-\U_\ell}^2+\norm{U_{0,\ell+1}-U_{0,\ell}}{H^{1/2}(\Gamma)}^2\big).
\end{align*}
Now, note $0<(1-q^{1/(d-1)})\theta<1$. Therefore, we may choose some sufficiently small $\delta>0$ such that $\rho:= (1+\delta)\big(1 -(1-q^{1/(d-1)})\theta\big) <1$. This concludes the proof of the estimate~\eqref{eq:estconv}.
\end{proof}

We recall that assumptions~(T2) and (E2) imply $\XX_{\ell}\subseteq\XX_{\ell+1}$.
The proof of convergence~\eqref{eq:convergence} of the adaptive coupling relies on the following a~priori
convergence result from~\cite[Proof of Proposition~10]{afp} which essentially follows from the nestedness $\XX_{\ell}\subseteq\XX_{\ell+1}$ of discrete spaces and
the validity of C\'ea's lemma~\eqref{eq:cea}.

\begin{lemma}\label{lemma:apriori}
According to assumption~(T2) and (E2), the limit $\U_\infty^\star:=\lim_\ell\U_\ell^\star$ of Galerkin solutions exists
strongly in $\HH$. If the limit $U_{0,\infty}:=\lim_\ell U_{0,\ell}$ exists strongly in
$H^{1/2}(\Gamma)$, then also the limit $\U_\infty:=\lim_\ell\U_\ell$ of the perturbed
Galerkin solutions exists strongly in $\HH$.\qed
\end{lemma}

Moreover, the following result from~\cite[Lemma~18]{fpp} proves {\sl a~priori} convergence
of $U_{0,\ell}$ for the Scott-Zhang projection. We remark that \cite[lemma~{18}]{fpp} is 
stated and proved for $p=1$ and $H^1(\Omega)$, but the proof holds without any modification
also \revision{for general $p\ge1$} and $H^1(\Gamma)$. We recall that assumption~(E2) implies
nestedness $\SS^p(\EE_\ell^\Gamma)\subseteq\SS^p(\EE_{\ell+1}^\Gamma)$, whereas (E1)
states uniform $\kappa$-shape regularity which enters the approximation and stability
estimates of the Scott-Zhang projection.

\begin{lemma}\label{lemma:scottzhang}
According to~(E1) and~(E2), the limit $\JJ_\infty g=\lim_\ell\JJ_\ell g$ of the Scott-Zhang
projections exists strongly in $H^1(\Gamma)$ for all $g\in H^1(\Gamma)$.\qed
\end{lemma}

\begin{proof}[Proof of Convergence~\eqref{eq:convergence} of Theorem~\ref{thm:convergence}]
According to Lemma~\ref{lemma:scottzhang}, the limit $U_{0,\infty} = \lim_\ell U_{0,\ell}$
exists strongly in $H^1(\Gamma)$ and hence also in $H^{1/2}(\Gamma)$.
Therefore, Lemma~\ref{lemma:apriori} applies
and proves that $\U_\infty:=\lim_\ell\U_\ell$ exists in $\HH$. Consequently,
the estimator reduction estimate~\eqref{eq:estconv} may be written as
\begin{align*}
 \varrho_{\ell+1}^2 \le \rho\,\varrho_\ell^2 + \alpha_\ell
 \quad\text{for all }\ell\ge0
\end{align*}
with a non-negative zero sequence $\alpha_\ell\to0$ as $\ell\to\infty$.
It is thus a consequence of elementary calculus that $\lim_\ell\varrho_\ell=0$,
cf.~\cite[Lemma~2.3]{estconv}. With $\varrho_\ell^2=\eta_\ell^2+\osc_\ell^2$
and reliability $\enorm{\u-\U_\ell}\lesssim\varrho_\ell$, we
conclude the proof.
\end{proof}

\begin{remark}
{\rm(i)} For $p=1$ and 2D, one may also use nodal interpolation $U_{0,\ell}=I_\ell u_0$ to discretize
the Dirichlet data. It is proved in~\cite[Proof of Proposition~5.2]{hh2mixed}
that the limit $U_{0,\infty}=\lim_\ell U_{0,\ell}$
then exists strongly in $H^1(\Gamma)$ and hence also in $H^{1/2}(\Gamma)$.\\
{\rm(ii)} For $p=1$ in 3D and newest vertex bisection, one may also use the $L^2$-orthogonal projection
$\Pi_\ell:L^2(\Gamma)\to\SS^1(\EE_\ell^\Gamma)$. It is proved in~\revision{\cite[Theorem~3]{kpp}} that $\Pi_\ell$ is $H^1$-stable and \revision{in~\cite[Lemma~12]{kp}} that the limit
$U_{0,\infty}=\lim_\ell U_{0,\ell}$ therefore exists weakly in $H^1(\Gamma)$ and, by interpolation,
hence strongly in $H^{1/2}(\Gamma)$.\\
{\rm(iii)} In either case, the claims~\eqref{eq:estconv}--\eqref{eq:convergence} of Theorem~\ref{thm:convergence} remain valid.
\end{remark}

\subsection{Extension to nonlinear transmission problems}\label{section:fembem:extension}
We consider the interface problem~\eqref{eq:strongform} with \emph{nonlinear}
$\frakA:\R^d\to\R^d$ and $\frakA\nabla u(x) := \frakA(\nabla u(x))$, see e.g.~\cite{cs1995,fembem}. 
Under monotonicity~\eqref{eq:A:elliptic} and Lipschitz 
continuity~\eqref{eq:A:lipschitz}, the coupling formulation~\eqref{eq:weakform}
as well as the discretizations~\eqref{eq:galerkin} and~\eqref{eq:perturbed}
still admit unique solutions. Moreover, the {\sl a~posteriori} error analysis 
of Proposition~\ref{prop:aposteriori} remains valid. Here, note that 
$\frakA\nabla U_\ell$ is $\TT_\ell$-piecewise Lipschitz continuous so that
all occurring terms as e.g.\ the jump terms~\eqref{eq:eta:jump}
are well-defined. 

While the convergence result of Theorem~\ref{thm:convergence} holds for
arbitrary polynomial degrees $p\ge1$ and $q\ge0$ in case of linear problems,
we did not succeed to prove the same for nonlinear 
transmission problem. Difficulties are only met in the first step
of the proof.
In~\eqref{eq0:rho:volume}, the 
term
\begin{align*}
 \sum_{T'\in\TT_{\ell+1}}|T'|^{2/d}\norm{\div(\frakA\nabla U_{\ell+1}-\frakA\nabla U_\ell)}{L^2(T')}^2
 \simeq \norm{h_\ell\,\div_\ell(\frakA\nabla U_{\ell+1}-\frakA\nabla U_\ell)}{L^2(\Omega)}^2
\end{align*}
arises, with $\div_\ell$ the $\TT_\ell$-piecewise divergence operator. However, for lowest-order volume elements $p=1$, this term vanishes.
Thus, our proof also covers the adaptive FEM-BEM coupling for certain nonlinear
coupling problems and lowest-order
elements $p=1$, where $\frakA:\R^d\to\R^d$ satisfies~\eqref{eq:A:elliptic}--\eqref{eq:A:lipschitz}.

For $p\ge1$ and $\frakA$ being linear, we used a scaling argument to bound the
latter term in~\eqref{eq0:rho:volume} by $\norm{\frakA}{W^{1,\infty}(\Omega)}\norm{\nabla U_{\ell+1}-\nabla U_\ell}{L^2(\Omega)}$.
For $\frakA:\R^d\to\R^d$ being nonlinear, we did neither succeed to find a similar bound (which seems, in general, to be impossible) nor to prove that this term vanishes as $\ell\to\infty$.

%====================================================================
\section{Numerical experiments}
%====================================================================
\label{section:numerics}
\def\err{\rm err}
%Figure: Z-shape
\begin{figure}[htbp]
\begin{center}
  \includegraphics[width=0.5\textwidth]{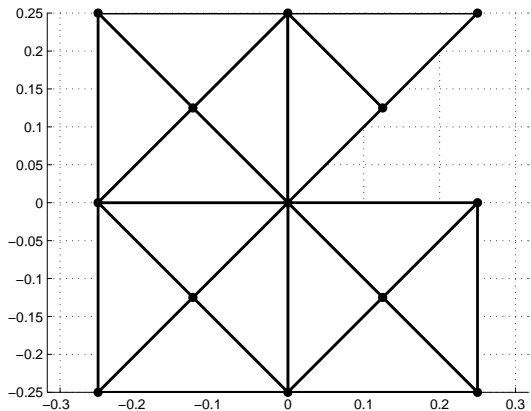}
  \caption{Z-shaped domain and initial triangulations $\TT_0$,  
  $\EE_0^\Gamma$ with $\#\TT_0 = 14$ triangles and \ $\#\EE_0^\Gamma = 10$
  boundary elements.}
  \label{fig:Zshape}
\end{center}
\end{figure}

%Figure: S2P1
\begin{figure}[htbp]
  \centering
  \begin{minipage}[b]{0.49\textwidth}
    \psfrag{convgraphResidual}[c][c]{}
    \psfrag{error}[c][c]{}
    \psfrag{nE}[c][c]{number of volume elements~$N$}
    \psfrag{etaAdap}{\tiny $\varrho_\ell(\Omega)$, adap}
    \psfrag{errAdap}{\tiny $\err_\ell(\Omega)$, adap}
    \psfrag{etaUnif}{\tiny $\varrho_\ell(\Omega)$, unif}
    \psfrag{errUnif}{\tiny $\err_\ell(\Omega)$, unif}

    \psfrag{OON1}[l][B]{\tiny $\OO(N^{-1})$}
    \psfrag{OON27}[l][B]{\tiny $\OO(N^{-2/7})$}

    \includegraphics[width=\textwidth]{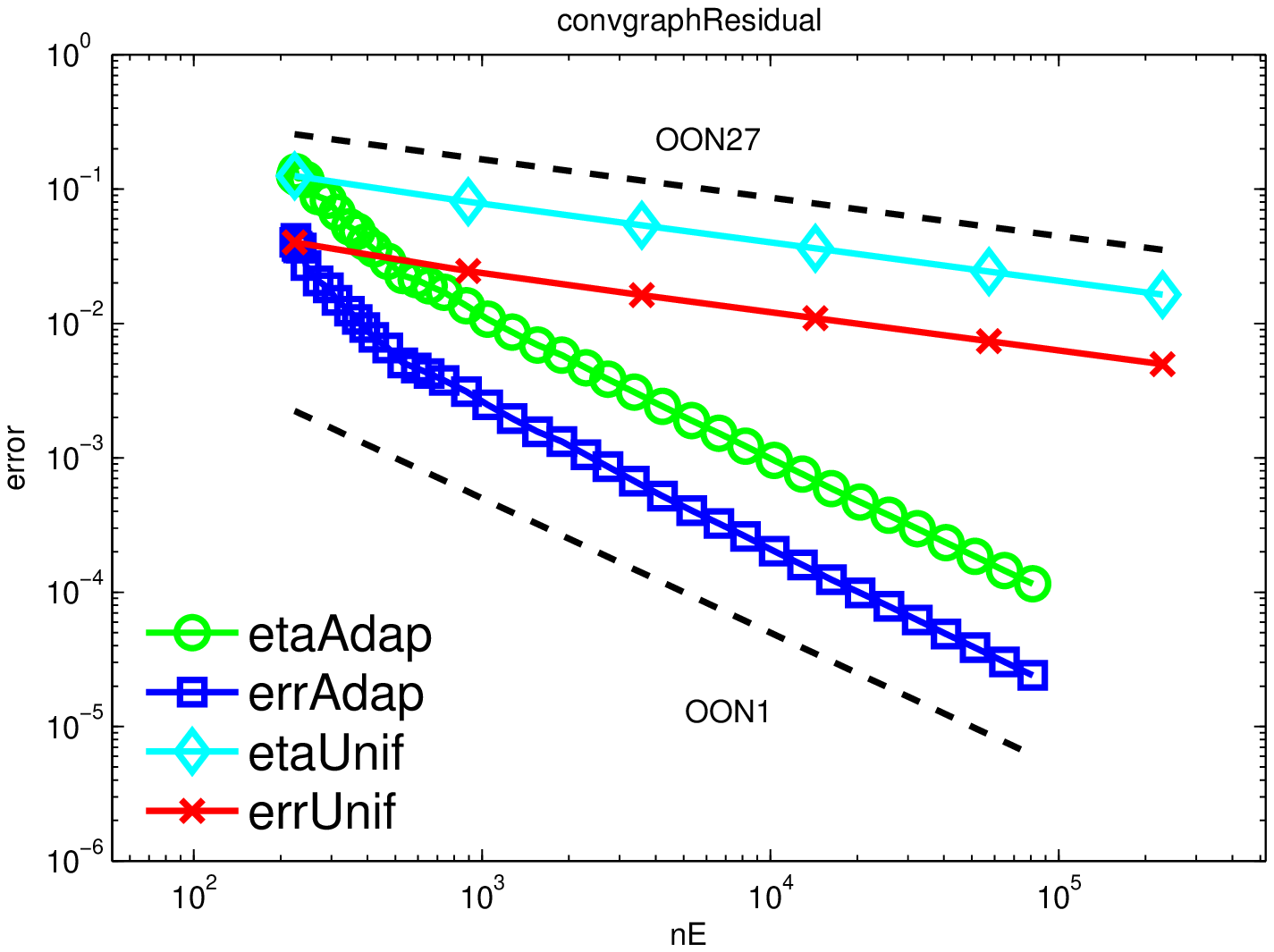}
  \end{minipage}
  \begin{minipage}[b]{0.49\textwidth}
    \psfrag{convgraphResidual}[c][c]{}
    \psfrag{error}[c][c]{}
    \psfrag{nB}[c][c]{number of boundary elements~$M$}
    \psfrag{etaAdap}{\tiny $\varrho_\ell(\Gamma)$, adap}
    \psfrag{errAdap}{\tiny $\err_\ell(\Gamma)$, adap}
    \psfrag{etaUnif}{\tiny $\varrho_\ell(\Gamma)$, unif}
    \psfrag{errUnif}{\tiny $\err_\ell(\Gamma)$, unif}

    \psfrag{OON23}[l][B]{\tiny $\OO(M^{-2/3})$}
    \psfrag{OON47}[l][B]{\tiny $\OO(M^{-4/7})$}
    \psfrag{OON52}[l][B]{\tiny $\OO(M^{-5/2})$}

    \includegraphics[width=\textwidth]{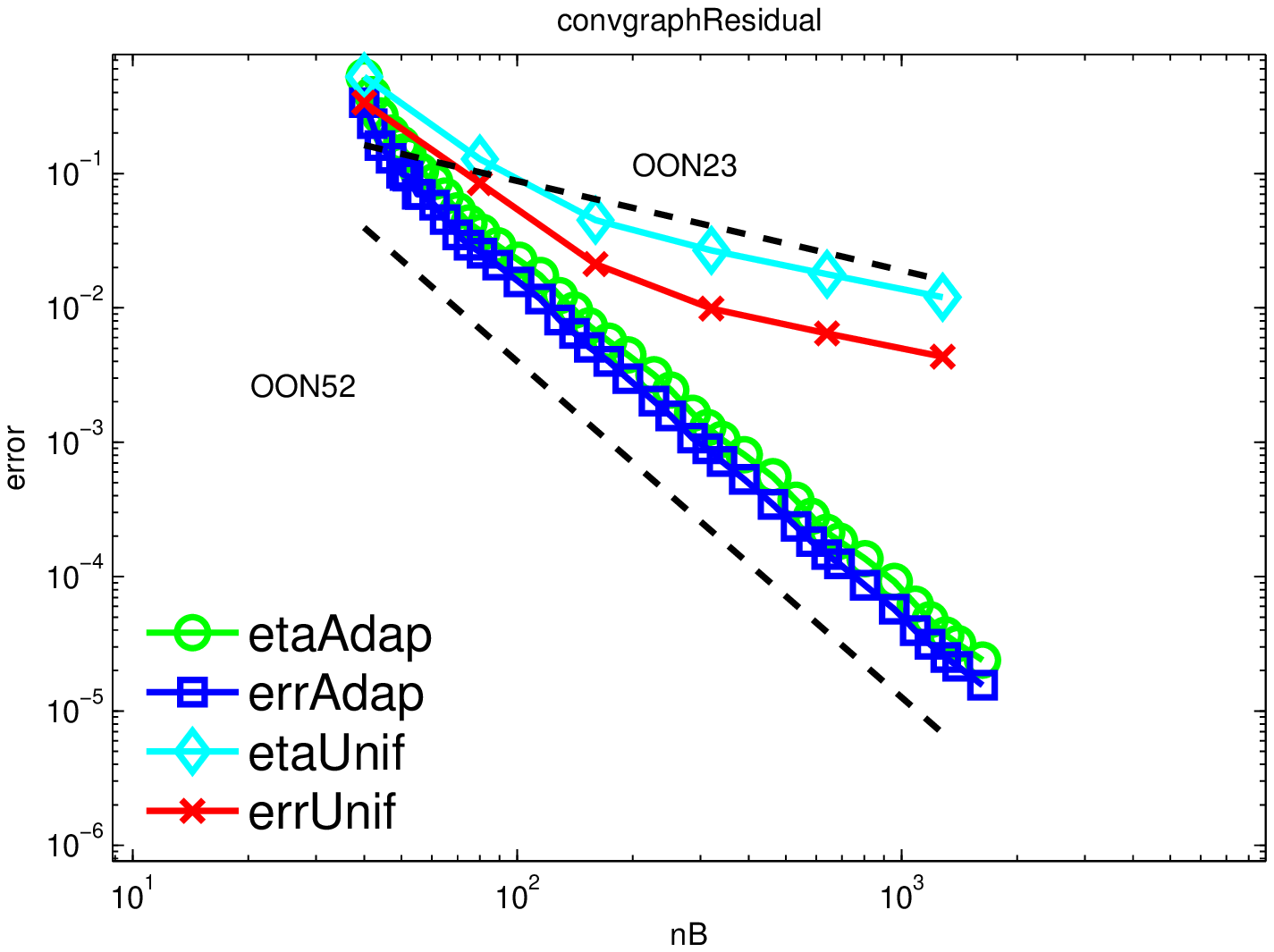}
  \end{minipage}
  \caption{$\err_\ell(\Omega)$ and $\varrho_\ell(\Omega)$ vs.\ $N=\#\TT_\ell$ (left) as well as $\err_\ell(\Gamma)$ and $\varrho_\ell(\Gamma)$ vs.\ $M=\#\EE_\ell^\Gamma$ for adaptive and uniform mesh refinement
with $p=2$ and $q=1$.}
  \label{fig:S2P1}
\end{figure}

%Figure: S2P0
\begin{figure}[htbp]
  \centering
  \begin{minipage}[b]{0.49\textwidth}
    \psfrag{convgraphResidual}[c][c]{}
    \psfrag{error}[c][c]{}
    \psfrag{nE}[c][c]{number of volume elements~$N$}
    \psfrag{etaAdap}{\tiny $\varrho_\ell(\Omega)$, adap}
    \psfrag{errAdap}{\tiny $\err_\ell(\Omega)$, adap}
    \psfrag{etaUnif}{\tiny $\varrho_\ell(\Omega)$, unif}
    \psfrag{errUnif}{\tiny $\err_\ell(\Omega)$, unif}

    \psfrag{OON1}[l][B]{\tiny $\OO(N^{-1})$}
    \psfrag{OON27}[l][B]{\tiny $\OO(N^{-2/7})$}

    \includegraphics[width=\textwidth]{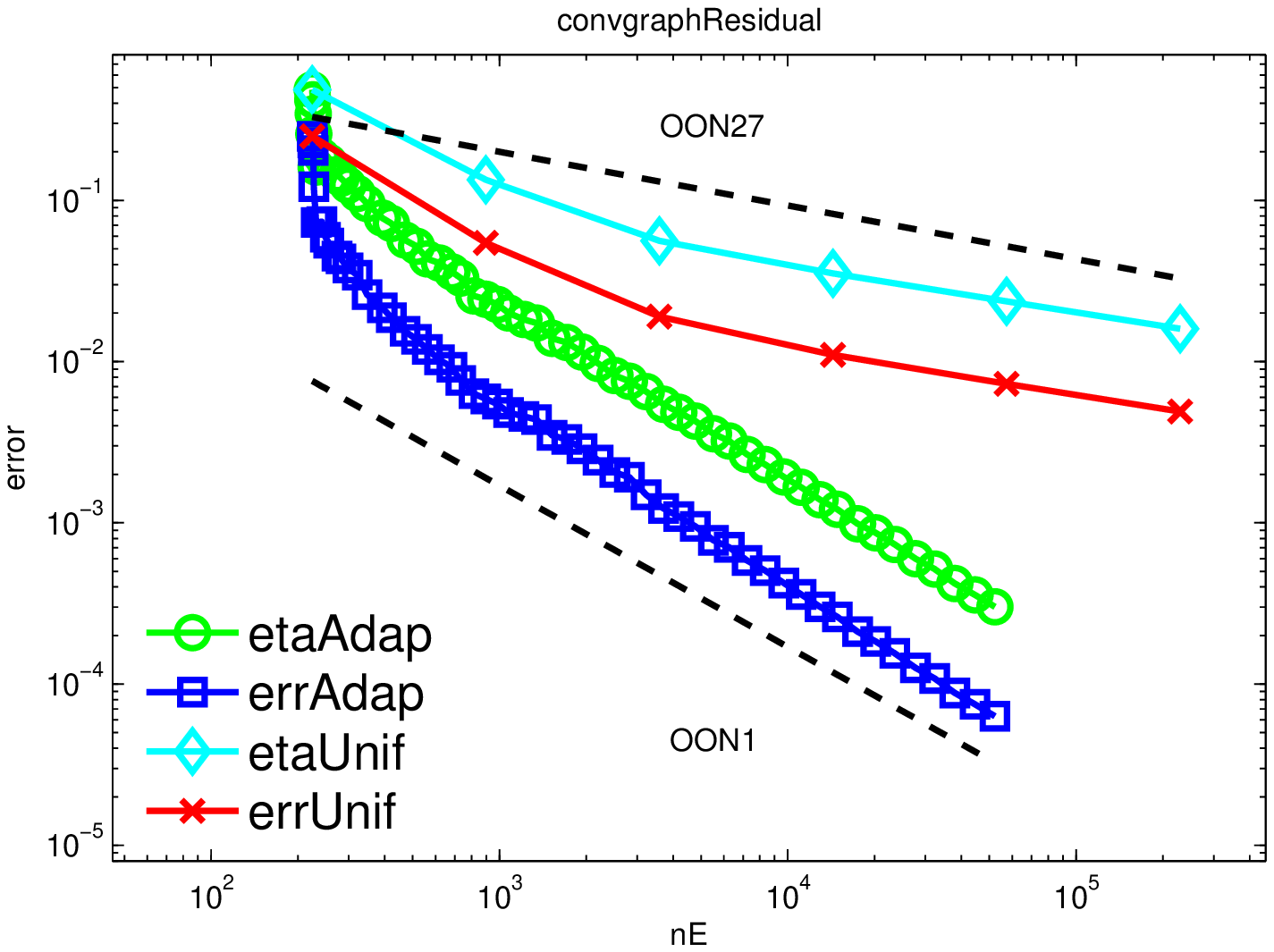}
  \end{minipage}
  \begin{minipage}[b]{0.49\textwidth}
    \psfrag{convgraphResidual}[c][c]{}
    \psfrag{error}[c][c]{}
    \psfrag{nB}[c][c]{number of boundary elements~$M$}
    \psfrag{etaAdap}{\tiny $\varrho_\ell(\Gamma)$, adap}
    \psfrag{errAdap}{\tiny $\err_\ell(\Gamma)$, adap}
    \psfrag{etaUnif}{\tiny $\varrho_\ell(\Gamma)$, unif}
    \psfrag{errUnif}{\tiny $\err_\ell(\Gamma)$, unif}

    \psfrag{OON23}[l][B]{\tiny $\OO(M^{-2/3})$}
    \psfrag{OON32}[l][B]{\tiny $\OO(M^{-3/2})$}

    \includegraphics[width=\textwidth]{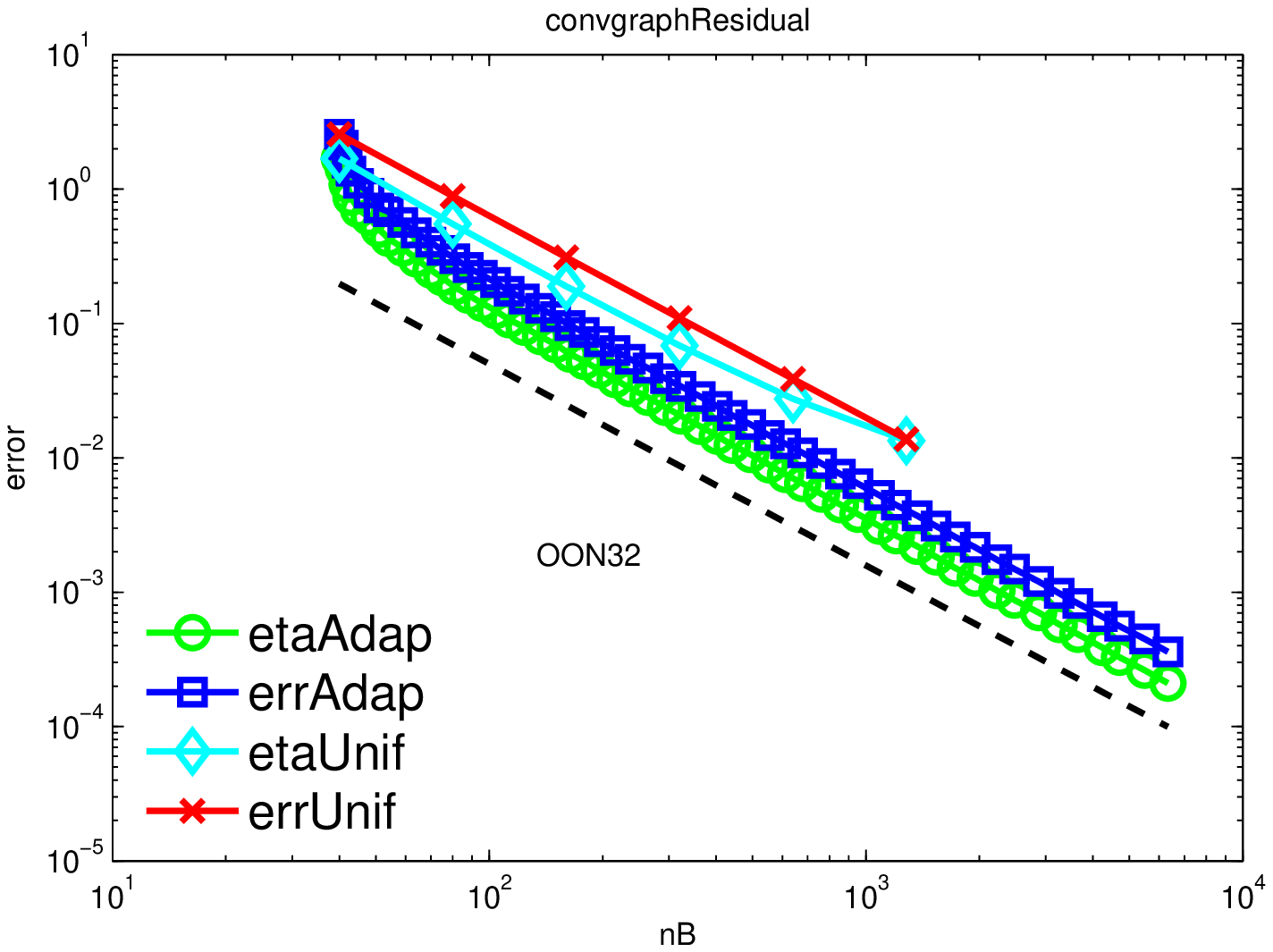}
  \end{minipage}
  \caption{$\err_\ell(\Omega)$ and $\varrho_\ell(\Omega)$ vs.\ $N=\#\TT_\ell$ (left) as well as $\err_\ell(\Gamma)$ and $\varrho_\ell(\Gamma)$ vs.\ $M=\#\EE_\ell^\Gamma$ for adaptive and uniform mesh refinement with $p=2$ and $q=0$.}
  \label{fig:S2P0}
\end{figure}

%Figure: S1P1
\begin{figure}[htbp]
  \centering
  \begin{minipage}[b]{0.49\textwidth}
    \psfrag{convgraphResidual}[c][c]{}
    \psfrag{error}[c][c]{}
    \psfrag{nE}[c][c]{number of volume elements~$N$}
    \psfrag{etaAdap}{\tiny $\varrho_\ell(\Omega)$, adap}
    \psfrag{errAdap}{\tiny $\err_\ell(\Omega)$, adap}
    \psfrag{etaUnif}{\tiny $\varrho_\ell(\Omega)$, unif}
    \psfrag{errUnif}{\tiny $\err_\ell(\Omega)$, unif}

    \psfrag{OON1}[l][B]{\tiny $\OO(N^{-1/2})$}
    \psfrag{OON27}[l][B]{\tiny $\OO(N^{-2/7})$}

    \includegraphics[width=\textwidth]{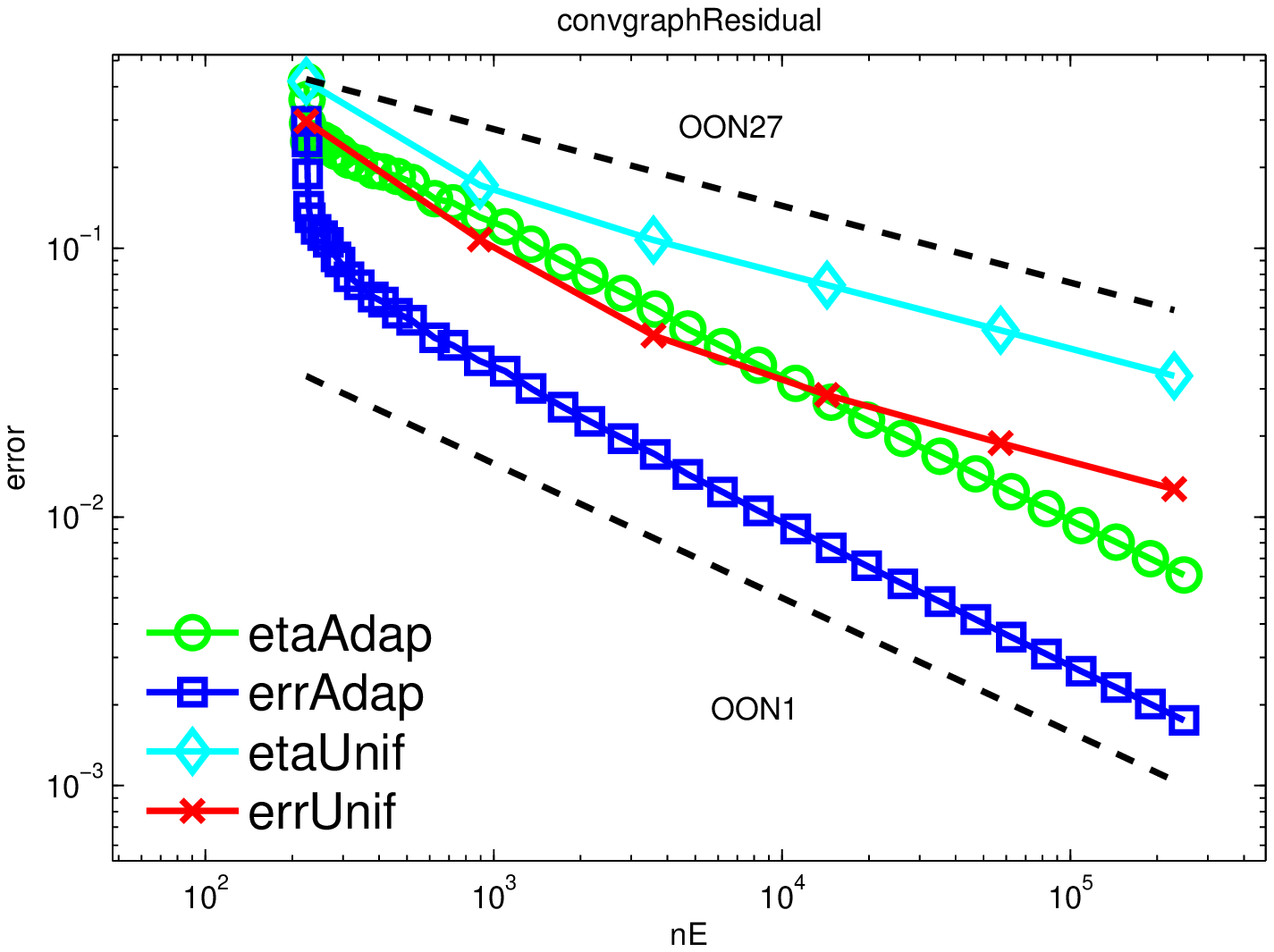}
  \end{minipage}
  \begin{minipage}[b]{0.49\textwidth}
    \psfrag{convgraphResidual}[c][c]{}
    \psfrag{error}[c][c]{}
    \psfrag{nB}[c][c]{number of boundary elements~$M$}
    \psfrag{etaAdap}{\tiny $\varrho_\ell(\Gamma)$, adap}
    \psfrag{errAdap}{\tiny $\err_\ell(\Gamma)$, adap}
    \psfrag{etaUnif}{\tiny $\varrho_\ell(\Gamma)$, unif}
    \psfrag{errUnif}{\tiny $\err_\ell(\Gamma)$, unif}

    \psfrag{OON12}[l][B]{\tiny $\OO(M^{-1.2})$}
    \psfrag{OON2}[l][B]{\tiny $\OO(M^{-2})$}

    \includegraphics[width=\textwidth]{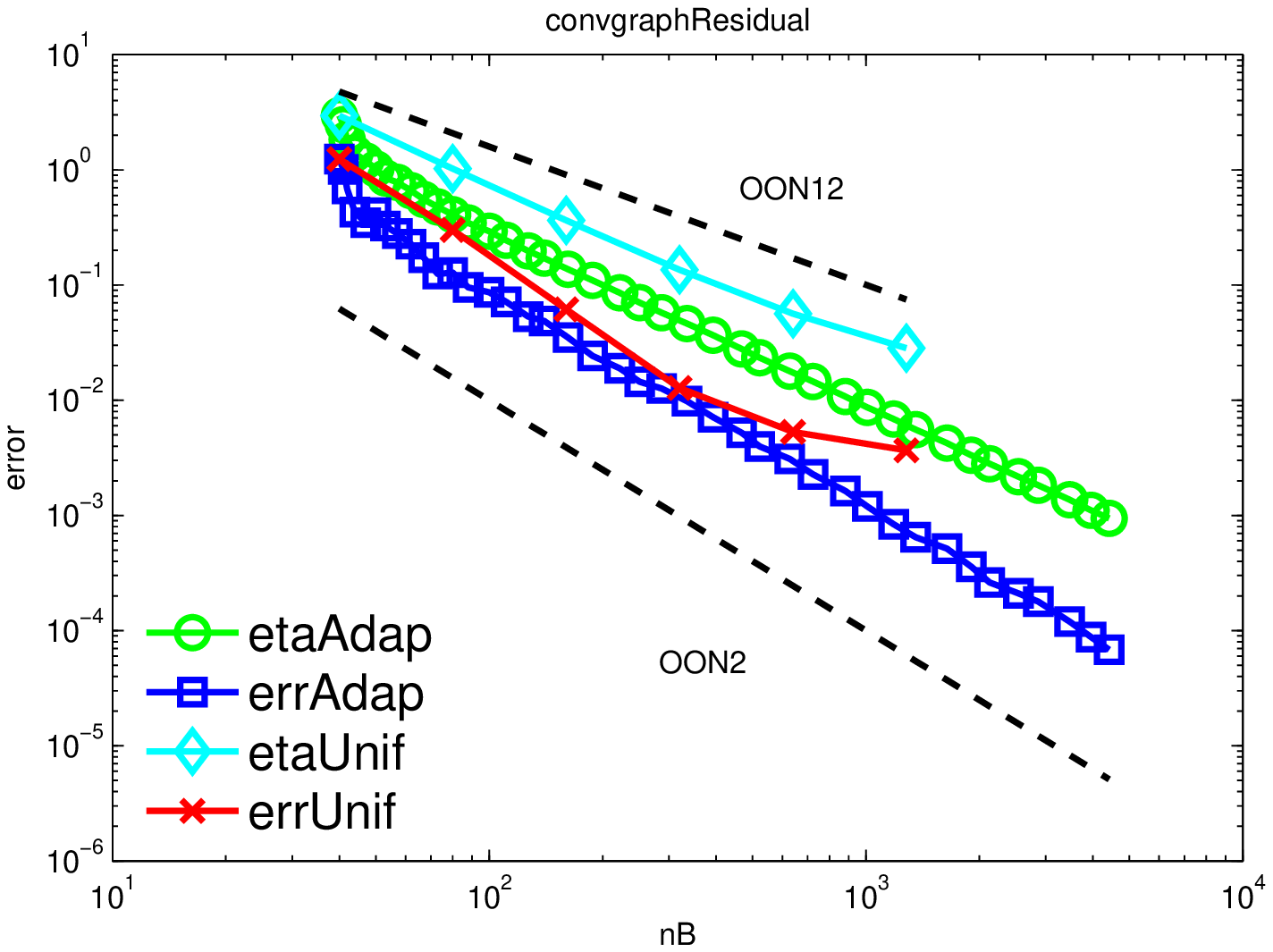}
  \end{minipage}
  \caption{$\err_\ell(\Omega)$ and $\varrho_\ell(\Omega)$ vs.\ $N=\#\TT_\ell$ (left) as well as $\err_\ell(\Gamma)$ and $\varrho_\ell(\Gamma)$ vs.\ $M=\#\EE_\ell^\Gamma$ for adaptive and uniform mesh refinement with $p=1$ and $q=1$.}
  \label{fig:S1P1}
\end{figure}

%Figure: S1P0
\begin{figure}[htbp]
  \centering
  \begin{minipage}[b]{0.49\textwidth}
    \psfrag{convgraphResidual}[c][c]{}
    \psfrag{error}[c][c]{}
    \psfrag{nE}[c][c]{number of volume elements~$N$}
    \psfrag{etaAdap}{\tiny $\varrho_\ell(\Omega)$, adap}
    \psfrag{errAdap}{\tiny $\err_\ell(\Omega)$, adap}
    \psfrag{etaUnif}{\tiny $\varrho_\ell(\Omega)$, unif}
    \psfrag{errUnif}{\tiny $\err_\ell(\Omega)$, unif}

    \psfrag{OON12}[l][B]{\tiny $\OO(N^{-1/2})$}
    \psfrag{OON27}[l][B]{\tiny $\OO(N^{-2/7})$}

    \includegraphics[width=\textwidth]{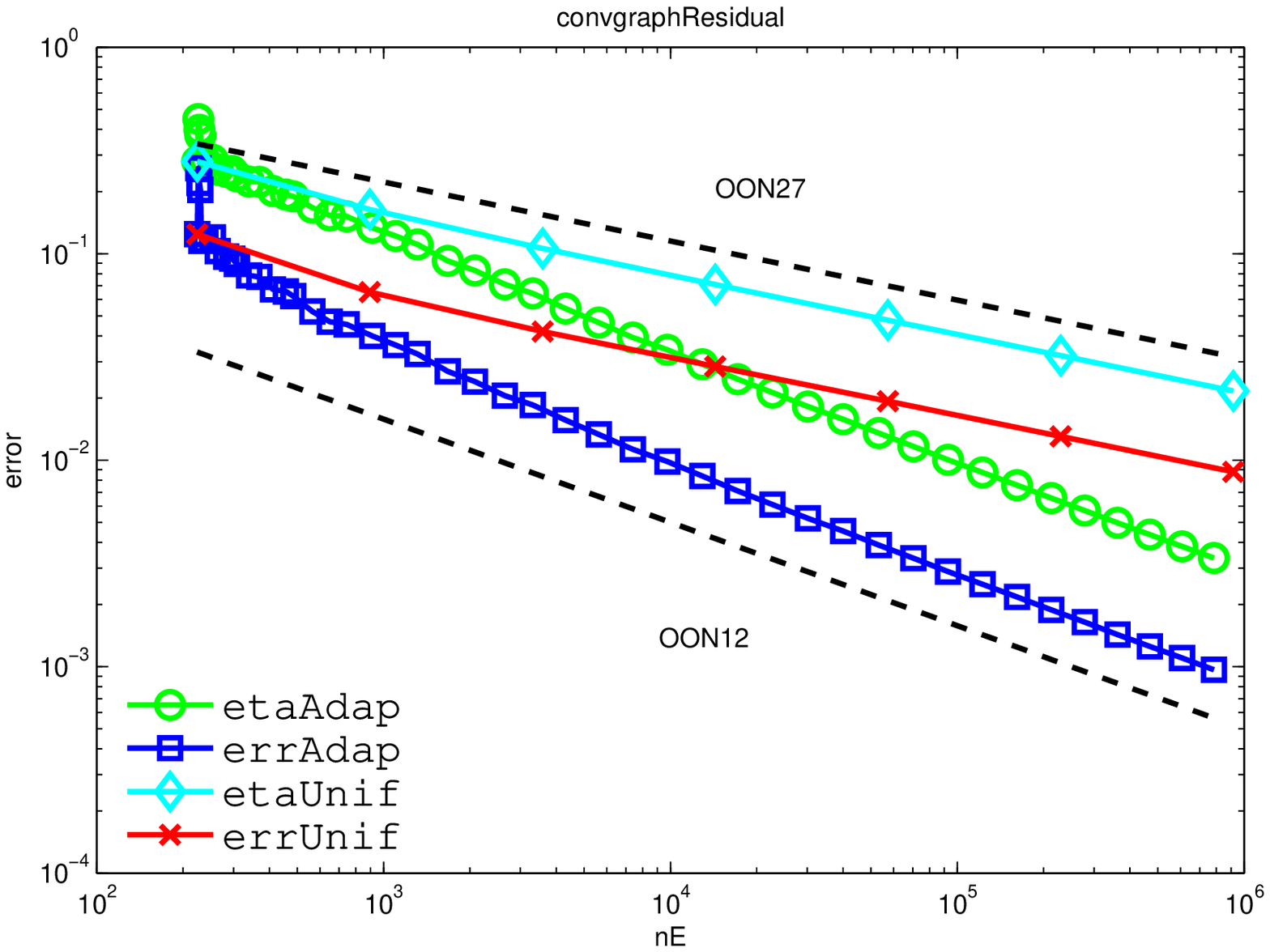}
  \end{minipage}
  \begin{minipage}[b]{0.49\textwidth}
    \psfrag{convgraphResidual}[c][c]{}
    \psfrag{error}[c][c]{}
    \psfrag{nB}[c][c]{number of boundary elements~$M$}
    \psfrag{etaAdap}{\tiny $\varrho_\ell(\Gamma)$, adap}
    \psfrag{errAdap}{\tiny $\err_\ell(\Gamma)$, adap}
    \psfrag{etaUnif}{\tiny $\varrho_\ell(\Gamma)$, unif}
    \psfrag{errUnif}{\tiny $\err_\ell(\Gamma)$, unif}

%    \psfrag{OON23}[l][B]{\tiny $\OO(M^{-2/3})$}
    \psfrag{OON32}[l][B]{\tiny $\OO(M^{-3/2})$}

    \includegraphics[width=\textwidth]{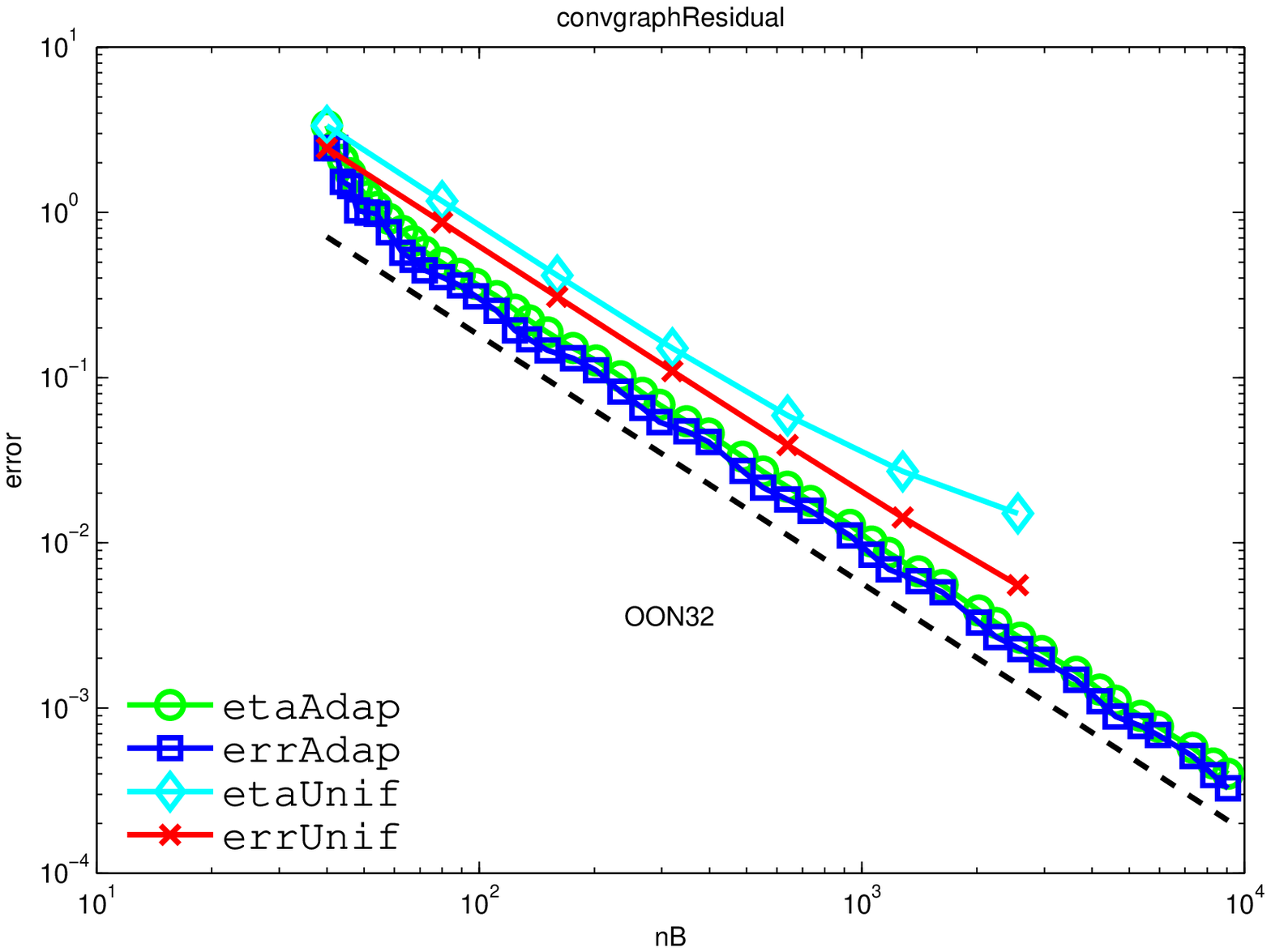}
  \end{minipage}
  \caption{$\err_\ell(\Omega)$ and $\varrho_\ell(\Omega)$ vs.\ $N=\#\TT_\ell$ (left) as well as $\err_\ell(\Gamma)$ and $\varrho_\ell(\Gamma)$ vs.\ $M=\#\EE_\ell^\Gamma$ for adaptive and uniform mesh refinement with $p=1$ and $q=0$.}
  \label{fig:S1P0}
\end{figure}
%==============================================================================
% NUMERIC SECTION
%==============================================================================
\noindent
In this section, we present 2D calculations for 
the perturbed Galerkin discretization~\eqref{eq:perturbed} of the model
problem~\eqref{eq:strongform}. We consider different choices of the polynomial
degrees $p$ and $q$ for the spaces $\SS^p(\TT_\ell)$, $\PP^q(\EE_\ell^\Gamma)$.
For ease of implementation and also for stability reasons, we restrict ourselves to
the case $\EE_\ell^\Gamma = \TT_\ell|_\Gamma$, i.e., the boundary mesh is taken as 
the trace of the volume triangulation. 
All computations were performed on a 64-BIT Intel(R) Core(TM) i7-3930K Linux
work station with 32GB of RAM. For the computation of the discrete boundary integral
operators we used the \textsc{Matlab} BEM-library \texttt{HILBERT}~\cite{hilbert},
and all systems of linear equations were solved with the \textsc{Matlab} backslash operator.

We employ the adaptive Algorithm~\ref{algorithm} with $\theta =0.25$ and compare the
results with uniform mesh refinement (this can be realized by setting $\theta = 1$ in 
Algorithm~\ref{algorithm}). The domain $\Omega$ is taken as the Z-shaped domain 
visualized in Fig.~\ref{fig:Zshape}. We take $\frakA = \rm Id$ in~\eqref{eq:strongform} 
and prescribe the data $f$, $u_0$, $\phi_0$ such that the exact solution is given by 
\begin{align*}
  u^{\rm int}(r,\varphi) &= r^{4/7}\sin(\tfrac47\varphi), \\
  u^{\rm ext}(x,y) &= \frac{x+y+0.25}{|x+\tfrac18|^2+|y+\tfrac18|^2};
\end{align*}
the polar coordinates $(r,\varphi)$ are taken with respect to the origin $(0,0)$. 
The prescribed solution $u^{\rm int}$ has the typical singularity
at the reentrant corner $(x,y) = (0,0)$, which leads to a reduced order of convergence
$\OO(h^{4/7})$ for uniform mesh refinement.

Recall that the {\sl a~posteriori} error estimator $\varrho_\ell$ from~\eqref{eq:indicators} is split into volume
contributions~\eqref{eq:eta:volume}--\eqref{eq:eta:jump} and boundary
contributions~\eqref{eq:eta:boundary}--\eqref{eq:osc}
\begin{align*}
  \varrho_\ell^2 = \varrho_\ell(\Omega)^2 + \varrho_\ell(\Gamma)^2.
\end{align*}
Arguing as in~\cite{afp}, the C\'ea-type quasi-optimality produces 
\begin{align*}
  \enorm{\u-\U_\ell}^2 = \norm{u-U_\ell}{H^1(\Omega)}^2 +
  \norm{\phi-\Phi_\ell}{H^{-1/2}(\Gamma)}^2 %\\
  &\lesssim  \norm{u-U_\ell}{H^1(\Omega)}^2 + \norm{h_\ell^{1/2}
  (\phi-\Phi_\ell)}{L^2(\Gamma)}^2 \\
  & =: \big(\err_\ell(\Omega)\big)^2 + \big(\err_\ell(\Gamma)\big)^2
\end{align*}
and thus provides a computable upper bound for the overall error.
We plot the error contribution $\err_\ell(\Omega)$ as well as the
corresponding estimator part $\varrho_\ell(\Omega)$ versus the
number of volume elements $N = \#\TT_\ell$. We also plot the
boundary contributions $\err_\ell(\Gamma)$ and $\varrho_\ell(\Gamma)$
versus the number of boundary elements $M = \#\EE_\ell^\Gamma$.
We observe convergence rates proportional to $N^{-\alpha}$ resp.\ $M^{-\beta}$ with
some $\alpha$, $\beta>0$ for all computed quantities.
The optimal convergence rate for the FEM part with $\SS^p(\TT_\ell)$ 
is $\alpha = p/2$, whereas the optimal rate for the BEM part with
$\PP^q(\EE_\ell^\Gamma)$ is $\beta = 3/2+q$.
Consequently, the optimal overall convergence rate for the FEM-BEM coupling is
proportional to $N^{-\alpha} + M^{-\beta}$.

The use of a uniform mesh refinement leads to suboptimal convergence rates for 
both the volume and boundary quantities. In
particular, we observe the rate $\alpha = 2/7$ independently of the
chosen polynomial order $p,q$, see
Figures~\ref{fig:S2P1},~\ref{fig:S2P0},~\ref{fig:S1P1}, and~\ref{fig:S1P0}.
We recall that for uniform mesh refinement there holds $N \simeq M^2$ and
therefore the overall convergence rate is proportional to $N^{-2/7}$.

In contrast to the uniform approach, the adaptive mesh refinement strategy
recovers the optimal overall convergence rate. We observe optimal rates $\alpha =
p/2$ and $\beta = 3/2 + q$ for $p=2$, $q=1$ as well as $p=2$, $q=0$ and $p=1$, $q=0$, see
Figures~\ref{fig:S2P1},~\ref{fig:S2P0}, and~\ref{fig:S1P0}.
In the case of $p=1,q=1$ we observe that the error estimator contribution
$\varrho_\ell(\Gamma)$ converges with order $M^{-3/2}$, whereas the error
quantity $\err_\ell(\Gamma)$ has some slightly higher convergence rate.
Nevertheless, we stress that the optimal overall convergence order is achieved,
since our computations also show that $N\simeq M^2$ in the case of $p=1,q=1$ and
therefore the overall error is dominated by the FEM part, which converges with
order $\alpha = 1/2$, see Figure~\ref{fig:S1P1}.

In conclusion, our numerical experiments underline the fact that the adaptive
mesh refinement strategy of Algorithm~\ref{algorithm}
achieves the optimal convergence rates in the presence of singularities of the exact solution.

\bigskip
\noindent
\textbf{Acknowledgement.}
The research of the authors M.~Aurada, M.~Feischl, M.~Karkulik, and D.~Praetorius is supported
through the FWF project
\emph{Adaptive Boundary Element Method}, funded by the Austrian Science Fund (FWF) under grant P21732.
The research of the author T.~F\"uhrer is additionally supported through the project
\emph{Effective Numerical Methods for the Johnson-N\'ed\'elec Coupling of FEM and BEM} funded by
the innovative projects initiative of Vienna University of Technology.

\newcommand{\bibentry}[2][]{\bibitem[#1]{#2}}
%\renewcommand{\bibentry}[2][]{\bibitem{#2}}
%====================================================================

\end{document}